\documentstyle[11pt,epsfig,mst-stylefile,amssymb,amsmath,harvard]{article}
\newcommand{\bbZ}{{\Bbb Z}}
\newcommand{\bbR}{{\Bbb R}}
\newcommand{\bbN}{{\Bbb N}}

\newcommand{\norm}[1]{\left\|#1\right\|}

\renewcommand{\cite}{\citeyear}

\begin{document}

\title{On the vaguelet and Riesz properties of $L^2$-unbounded \\ transformations
of orthogonal wavelet bases
\thanks{The first author was supported in part by the Louisiana Board of Regents award LEQSF(2008-11)-RD-A-23. The third author was supported in part by the NSF grant DMS-0505628.}
 }

\author{Gustavo Didier\\ Tulane University \and  St\'{e}phane Jaffard \\ Universit\'{e} Paris-Est \\ Cr\'{e}teil Val-de-Marne  \and
Vladas Pipiras \\ University of North Carolina}

\bibliographystyle{agsm}

\maketitle

\begin{abstract}
In this work, we prove that certain $L^2$-unbounded transformations of orthogonal wavelet bases generate vaguelets. The $L^2$-unbounded functions involved in the transformations are assumed to be quasi-homogeneous at high frequencies. We provide natural examples of functions which are not quasi-homogeneous and for which the resulting transformations are not vaguelets. We also address the related question of whether the considered family of functions is a Riesz basis in $L^2(\bbR)$. The Riesz property could be deduced directly from the results available in the literature or, as we outline, by using the vaguelet property in the context of this work. The considered families of functions arise in wavelet-based decompositions of stochastic processes with uncorrelated coefficients.
\end{abstract}


\section{Introduction}
\label{s:intro}

A family of functions $\{f_{j,k}\}$ is called {\it vaguelets} (Donoho \cite{donoho:1995}, Coifman and Meyer \cite{coifman:meyer:1997}, p.\ 56) when they satisfy
\begin{equation}\label{e:Meyer_Coif_expression_5.1}
|f_{j,k}(t)| \leq C 2^{j/2} ( 1 + |2^j t - k|)^{-1 - \alpha_1} ,
\end{equation}
\begin{equation}\label{e:Meyer_Coif_expression_5.2}
\int_{\bbR}f_{j,k}(t)dt = 0,
\end{equation}
\begin{equation}\label{e:Meyer_Coif_expression_5.3}
|f_{j,k}(t)- f_{j,k}(t')| \leq C 2^{j(1/2 + \alpha_2)} |t- t'|^{\alpha_2},
\end{equation}
where the constant $C$ does not depend on $j$, $k$ and
\begin{equation}\label{e:cond_alpha_beta}
0 < \alpha_2 < \alpha_1 < 1.
\end{equation}
Property \eqref{e:Meyer_Coif_expression_5.1} describes the time localization of $f_{j,k}$. Property \eqref{e:Meyer_Coif_expression_5.2} sets the mean of $f_{j,k}$ to zero, thus implying that $f_{j,k}$ oscillates. Property \eqref{e:Meyer_Coif_expression_5.3} describes the regularity of $f_{j,k}$. Vaguelets share many properties with a standard wavelet basis. For instance, they are both indexed by scale and shift. Moreover, in the space $L^2(\bbR)$, vaguelets satisfy the useful inequality
\begin{equation}\label{e:bounds_vaguelets}
\Big\| \sum_{j,k} d_{j,k} f_{j,k} \Big\|^2_{L^2(\bbR)} \leq C \sum_{j,k} d^2_{j,k}
\end{equation}
for some $C > 0$ (e.g., Theorem 2 on p.\ 56 of Coifman and Meyer \cite{coifman:meyer:1997}). The bound \eqref{e:bounds_vaguelets} brings vaguelets close to Riesz bases. Recall that a family $\{e_l\}_{l \in \bbZ}$ of elements of a Hilbert space $\mathcal{H}$ is called a \textit{Riesz basis} if
\begin{itemize}
\item [$(i)$] there exist constants $C_2 \geq C_1 > 0$ such that, for all sequences $\{a_{l}\}_{l \in \bbZ} \in l^2(\bbZ)$,
\begin{equation} \label{e:Riesz_basis_equation_on_bounds}
C_1 \Big(\sum_{l \in \bbZ} |a_{l}|^2 \Big)^{1/2} \leq
\Big\|\sum_{l \in \bbZ} a_{l} e_{l}\Big\|_{\mathcal{H}} \leq
C_2 \Big(\sum_{l \in \bbZ} |a_{l}|^2 \Big)^{1/2};
\end{equation}
\item [$(ii)$] span$\{e_{l}\}_{l \in \bbZ}$ is dense in $\mathcal{H}$.
\end{itemize}
Apart from its natural place in Functional Analysis, vaguelets have also found applications, for instance, in Statistics (Donoho \cite{donoho:1995}).

Our focus is on the vaguelet property of the family $\{\Psi_{j,k}, j \geq 0, k \in \bbZ\} $, where
\begin{equation}\label{e:Psi_jk}
\Psi_{j,k} = \frac{\Psi^{\#}_{j,k}}{\| \Psi^{\#}_{j,k} \|}.
\end{equation}
Here and throughout, $\| \cdot \|$ denotes the usual $L^2(\bbR)$ norm and the functions $\Psi^{\#}_{j,k}$ are defined in the Fourier domain as
\begin{equation}\label{e:psihat_uppersharp}
\widehat{\Psi}^{\#}_{j,k}(x)= h_2(x) \widehat{\psi}_{j,k}(x),
\end{equation}
where $h_2$ is an appropriate complex-valued function and $\psi_{j,k}(t) = 2^{j/2}\psi(2^j t - k)$ is a scaled and translated copy of a wavelet $\psi$ (see, for example, Daubechies \cite{daubechies:1992}, Meyer \cite{meyer:1992}, Mallat \cite{mallat:1998}). We use the convention $\widehat{f}(x) = \int_{\bbR} e^{-ixt}f(t) dt$ for the Fourier transform of $f \in L^2(\bbR)$. The precise statement of the assumptions used in this paper can be found in Section \ref{s:assumptions-result}. In particular, it is \textit{not} assumed that $h_2, h^{-1}_2 $ are functions in $L^\infty(\bbR)$ or even $L^2(\bbR)$.

Starting from appropriate conditions on the functions $h_2$ and $\psi$, the proof that the functions
\begin{equation}\label{e:psi_jk_as_vaguelets}
\{\Psi_{j,k}, j \geq 0, k \in \bbZ \}
\end{equation}
 are vaguelets makes use of a standard argument in the Fourier domain (Section \ref{s:proof-vaguelet}). We suppose, in particular, that the function $h_2$ is \textit{quasi-homogeneous} in the sense that, for some $d \in \bbR$ and any $\epsilon > 0$, there are constants $C_2 \geq C_1 > 0$ (depending on $\epsilon$) such that
\begin{equation}\label{e:quasihomogeneity}
C_1 |x|^{-d} \leq |h_2(x)| \leq C_2 |x|^{-d}
\end{equation}
for $|x|> \epsilon$. Similar assumptions have been made in the works of Donoho \cite{donoho:1995} and others. In a more original direction, we also shed light on the necessity of the quasi-homogeneity assumption \eqref{e:quasihomogeneity}. We provide natural examples of functions $h_2$ which are not quasi-homogeneous and for which the resulting family are not vaguelets (Section \ref{s:necessity_frac_decay}).

We are also interested in whether the vaguelet family \eqref{e:psi_jk_as_vaguelets} leads to a Riesz basis in $L^2(\bbR)$. In order to generate a basis in $L^2(\bbR)$, we need to complement the family \eqref{e:psi_jk_as_vaguelets} with functions playing the role of approximation functions associated with $j=0$. Let
\begin{equation}\label{e:Phi_jk}
\Phi_{j,k} = \frac{\Phi^{\#}_{j,k}}{\| \Phi^{\#}_{j,k}\|}, \quad \widehat{\Phi}^{\#}_{j,k}(x) = h_1(x)\widehat{\phi}_{j,k}(x),
\end{equation}
where $h_1$ is a complex-valued function and $\phi_{j,k}(t) = 2^{j/2}\phi(2^j t - k)$ is a scaled and translated copy of a scaling function $\phi$. We are interested in the Riesz property (in $L^2(\bbR)$) of the family of functions
\begin{equation}\label{e:F_0}
\{\Phi_{0,k},\Psi_{j,k}, j \geq 0, k \in \bbZ\}.
\end{equation}
Establishing the Riesz property of the functions \eqref{e:F_0} is closely related to doing so for the ``dual" family
\begin{equation}\label{e:F^0}
\{\Phi^{0,k},\Psi^{j,k}, j \geq 0, k \in \bbZ\}
\end{equation}
with the ``dual" functions
\begin{equation}\label{e:biorth-func-transform-normalized}
 \Phi^{j,k} = \frac{ \Phi^{j,k}_\#}{\|
    \Phi^{j,k}_\#\|},\quad  \Psi^{j,k} = \frac{\Psi^{j,k}_\#}{\|
    \Psi^{j,k}_\#\|},
\end{equation}
where
\begin{equation}\label{e:biorth-func-transform}
    \widehat \Phi^{j,k}_\#(x) = \overline{h_1(x)^{-1}} \widehat \phi_{j,k}(x),\quad
    \widehat \Psi^{j,k}_\#(x) = \overline{h_2(x)^{-1}} \widehat \psi_{j,k}(x).
\end{equation}
Under the assumption that $h_2/h_1$ is $2\pi$-periodic and some other mild assumptions, the functions in \eqref{e:F_0} and \eqref{e:F^0} satisfy the biorthogonality relations
\begin{equation}\label{e:biortho_relation}
\int_{\bbR} \Phi^{0,k}(t) \Phi_{0,k'}(t) dt = \delta_{\{k=k'\}}, \quad
\int_{\bbR} \Psi^{j,k}(t) \Psi_{j',k'}(t) dt = \delta_{\{j=j',k=k'\}},
\end{equation}
as well as
\begin{equation}\label{e:orthog_of_Phi_and_Psi}
\int_{\bbR} \Phi^{0,k}(t) \Psi_{j,k'}(t) dt = 0,
\end{equation}
where $j',j \geq 0$, $k,k' \in \bbZ$. Our discussion of the Riesz property will involve both families of functions \eqref{e:F_0} and \eqref{e:F^0}.

The Riesz property for the families \eqref{e:F_0} and \eqref{e:F^0} can be deduced directly from the results on the Riesz property of orthogonal wavelet bases in Sobolev spaces. The latter has been a quite active research topic (see Han and Shen \cite{han:shen:2009} and references therein). We discuss the connection with the Riesz property in Sobolev spaces in Section \ref{s:Sobolev}. As pointed out above, vaguelets are naturally related to Riesz bases. We thus also outline how the Riesz property for \eqref{e:F_0} and \eqref{e:F^0} can be obtained in the vaguelet framework of this paper, without resorting to the results on Sobolev spaces (Section \ref{s:proof-riesz}).

This work was motivated by the fact that the bases \eqref{e:F_0} and \eqref{e:F^0} arise in wavelet decompositions of stochastic processes with uncorrelated wavelet coefficients. In the case of zero-mean stationary processes $X = \{X(t)\}_{t \in \bbR}$, bases of the form \eqref{e:F_0} and \eqref{e:F^0} give rise to uncorrelated wavelet coefficients as long as $|h_2(x)|^2$ is the Fourier transform of the autocovariance function $EX(0)X(u) = \gamma(u)$ (Zhang and Walter \cite{zhang:walter:1994}, Benassi, Jaffard and Roux \cite{benassi:jaffard:roux:1997}, Ruiz-Medina, Angulo and Anh \cite{ruiz-medina:angulo:anh:2003}). Moreover, the flexibility of taking $h_1$ different from $h_2$ has proved useful for constructing expansions of stochastic processes with \textit{correlated }approximation coefficients (and uncorrelated wavelet coefficients), which in turn were used for simulation (Meyer et al.\ \cite{MST:1999}, Pipiras \cite{pipiras:2005}, Didier and Pipiras \cite{didier:pipiras:2008}, Didier and Fricks \cite{didier:fricks:2012}). More specifically, in the Gaussian case such expansions have the general form
$$
X(t) = \sum^{\infty}_{k = -\infty} a_{J,k}  \Phi^{J,k} (t) +
\sum^{\infty}_{j=J} \sum^{\infty}_{k = -\infty} d_{j,k}  \Psi^{j,k}(t), \quad J \in \bbZ,
$$
where $\{ d_{j,k}\}_{j \geq 0, k \in \bbZ}$ are independent standard Normal random variables, but $\{ a_{J,k} \}_{k \in \bbZ}$ are allowed to be correlated. Zhang and Walter \cite{zhang:walter:1994}, p.\ 1742, were the first to conjecture that the families of functions \eqref{e:F_0} and \eqref{e:F^0} are Riesz bases, but the question seems to have remained open ever since. Benassi and Jaffard \cite{benassi:jaffard:1994} display some results related to the vaguelet property but the question is not raised or addressed explicitly. In a related direction, Meyer et al.\ \cite{MST:1999} considered an important process with stationary increments called fractional Brownian motion. These authors established an analogous wavelet decomposition with uncorrelated wavelet coefficients by taking
\begin{equation}\label{e:h1_h2_MST}
h_1(x) = \Big(\frac{ix}{1 - e^{-ix}}\Big)^{d}, \quad h_2(x) = (ix)^{d},
\end{equation}
where $(ix)^{d} = |x|^{d} e^{i \pi d/2}$ if $x > 0$, and $= |x|^{d} e^{-i\pi d/2}$ if $x < 0$. For this choice of $h_1,h_2$ and the underlying Meyer wavelet basis, Meyer et al.\ \cite{MST:1999}, Theorem 1, showed that the resulting family of functions forms a Riesz basis as well as a vaguelet family. Some related work can be found in Unser and Blu \cite{unser:blu:2005cardinal1}, Blu and Unser \cite{blu:unser:2007:self2}, Tafti, Van De Ville and Unser \cite{tafti:deville:unser:2009}.

In the context of wavelet decomposition of stochastic processes described above, it is natural to ask whether the bases are vaguelets and Riesz. For example, a Riesz basis for $L^2(\bbR)$ can be used in a multiresolution analysis (see, for instance, Daubechies \cite{daubechies:1992}, p.\ 139, and Meyer et al.\ \cite{MST:1999}). The vaguelet property paves the way for proving local and global regularity and irregularity results for the associated stochastic processes. For the case of fractional Brownian motion and the special wavelet basis obtained from (\ref{e:h1_h2_MST}), see Jaffard, Lashermes and Abry \cite{jaffard:2007} and, in particular, their Theorem 5.1. With a view towards simulation, constructing vaguelet families is efficient and useful in the sense that it allows for simulating the process in different time regions with different accuracies.

The rest of the paper is organized as follows. In Section \ref{s:assumptions-result}, we introduce further notation and the assumptions. In Section \ref{s:vaguelet}, we consider the vaguelet property of the families $\{\Psi_{j,k}, j \geq 0, k \in \bbZ \}$ and $\{\Psi^{j,k}, j \geq 0, k \in \bbZ \}$. Section \ref{s:riesz} contains the discussions and results on the Riesz property of the families \eqref{e:F_0} and \eqref{e:F^0}.

\section{Assumptions and statements of the main results}
\label{s:assumptions-result}

The transformations induced by \eqref{e:Psi_jk}, \eqref{e:Phi_jk} and \eqref{e:biorth-func-transform-normalized} are characterized by the following assumptions.

\medskip
{\sc Assumption $(h1)$:}
\begin{equation}\label{e:assumption_h1_h1inv_bounded_compactinterv}
\textnormal{$h_1(x) = \overline{h_1(-x)}$ and $h_1, h_1^{-1}$ are bounded on
every compact interval of $\bbR$};
\end{equation}

\medskip
{\sc Assumption $(h2)$:}
\smallskip

$(a)$
\begin{equation}\label{e:h2,h2inv_in_C2}
h_2(x) = \overline{h_2(-x)} \textnormal{ and }h_2, h^{-1}_2 \in C^2(\bbR);
\end{equation}

$(b)$ for some $d \in \bbR$ and any $\epsilon > 0$, there are constants $C_2 \geq C_1 > 0$ and $C_3 > 0$ (depending on $\epsilon$) such that
\begin{equation}\label{e:bounds-h_2}
C_1 |x|^{-d} \leq |h_2(x)| \leq C_2 |x|^{-d}, \quad \textnormal{for}
\hspace{0.1cm} |x|> \epsilon,
\end{equation}
\begin{equation}\label{e:assumption_decay_h2_h2inv}
\Big|\frac{d^{k}}{dx^{k}}(h_2)^p(x) \Big| \leq C_3 |x|^{-pd-k}, \quad p=-1,1, \quad k=1,2, \quad |x| > \epsilon,
\end{equation}
\begin{equation}\label{e:assumption_local_scaling}
\max_{p=-1,1}\max_{k=0,1,2} \frac{1}{2^{-pjd}} 2^{jk}\Big|\frac{d^{k}}{dx^{k}}(h_2)^{p}(2^j x) \Big| \leq \frac{C_3}{|x|^{N-2}}, \quad j \geq 1, \quad |x| \leq \epsilon,
\end{equation}
for some $N \geq 2$ (this will be the same $N \in \bbN$ as in condition \eqref{e:general_assumption_decay_MRA} below);
%

\medskip
{\sc Assumption $(h3)$:}
\begin{equation}\label{e:assumption_h2/h1_is_periodic}
\textnormal{The function $h_2/h_1$ is periodic with period
$2\pi$.}
\end{equation}
\begin{example}\label{ex:rational_functions}
All rational functions $h_2(x) = \frac{p(x)}{q(x)}$, where $p$ and $q$ are polynomials with complex coefficients and $q$ does not vanish a.e., satisfy ($\ref{e:bounds-h_2}$). The particular instance
\begin{equation}\label{e:h2_OU_0}
h_2(x) = \frac{1}{1+ix},
\end{equation}
which fits under ($\ref{e:bounds-h_2}$) with $d = 1$, is associated in Didier and Pipiras \cite{didier:pipiras:2008} with the Ornstein-Uhlenbeck process. This is so because the inverse Fourier transform of $|h_2(x)|^2$ gives the autocovariance function of the latter. Equivalently, instead of \eqref{e:h2_OU_0} one can take the expression
\begin{equation}\label{e:h2_OU}
h_2(x) = \frac{1}{\sqrt{1+x^2}}
\end{equation}
(see Example \ref{ex:h2_OU_satisfies_the_assumptions} below).

\end{example}
\begin{remark}\label{r:exp(-x^2)}
The condition (\ref{e:bounds-h_2}) does not encompass the seemingly simple, or simpler, case where $h_2(x) = e^{-x^2}$. This is so for a good reason, as shown in Section \ref{s:necessity_frac_decay}.
\end{remark}


\begin{remark}
All through the paper, we will use either $C$ or indexed $C$ to denote a constant that may change from line to line, but which does not depend on the wavelet indices $j$ or $k$.
\end{remark}

\begin{example}\label{ex:h2_OU_satisfies_the_assumptions}
The function $h_2$ in \eqref{e:h2_OU} satisfies the condition \eqref{e:assumption_local_scaling} with $N \geq 3$. In fact, for $h_2$ with $d = 1$,
$$
\frac{1}{2^{-j}}|h_2(2^jx)| = \frac{1}{2^{-j}}\frac{1}{\sqrt{1 + 2^{2j} x^2}} = \frac{1}{\sqrt{2^{-2j} + x^2}}
\leq \frac{1}{|x|} \leq \frac{C}{|x|^{N-2}}, \quad |x| \leq \epsilon.
$$
As for the derivatives when $|x|\leq \epsilon$, denote by $\frac{d}{dx}h_2(2^j x)$ the function $\frac{d}{dx}h_2$ evaluated at $2^j x$ (and not the derivative of $h_2(2^jx)$ with respect to $x$). We obtain
$$
\frac{2^j}{2^{-j}}\Big| \frac{d}{dx}h_2(2^{j}x) \Big| = \frac{2^j}{2^{-j}}  (1 + 2^{2j}x^2)^{-3/2}|2^j x| = \Big( \frac{1}{2^{2j}} + x^2\Big)^{-3/2} |x| \leq \frac{1}{|x|^2} \leq \frac{C}{|x|^{N-1}},
$$
$$
\frac{2^{2j}}{2^{-j}}\Big| \frac{d^2}{dx^2} h_2(2^{j}x)\Big| \leq \frac{2^{2j}}{2^{-j}}  |3(1 + 2^{2j}x^2)^{-5/2}2^{2j}x^2 - (1 + 2^{2j}x^2)^{-3/2}| \leq \frac{3x^2}{|x|^5} + \frac{1}{|x|^3} \leq \frac{C}{|x|^N}.
$$
Similarly, for $h^{-1}_2$,
$$
\frac{1}{2^j} |h^{-1}_2(2^jx)|= \frac{1}{2^j}(1 + 2^{2j}x^2)^{1/2} = \Big(\frac{1}{2^{2j}} + x^2 \Big)^{1/2} \leq (1 + x^2 )^{1/2} \leq \frac{1}{|x|^{N-2}},
$$
$$
\frac{2^{j}}{2^j}\Big| \frac{d}{dx}h^{-1}_2(2^{j}x) \Big| = (1 + 2^{2j}x^2)^{-1/2}|2^j x| = \Big(\frac{1}{2^{2j}} + x^2\Big)^{-1/2}|x| \leq \frac{1}{|x|}|x| \leq \frac{C}{|x|^{N-1}},
$$
and
$$
\frac{2^{2j}}{2^{j}}\Big| \frac{d^2}{dx^2} h^{-1}_2(2^{j}x)\Big| \leq \Big| 2^j [(-1)(1 + 2^{2j}x^2)^{-3/2}2^{2j}x^2 + (1 + 2^{2j}x^2)^{-1/2}] \Big|
$$
$$
\leq 2^j \Big[ 2^{-j}\Big( \frac{1}{2^{2j}} + x^2 \Big)^{-3/2}x^2 + 2^{-j} \Big( \frac{1}{2^{2j}} + x^2\Big)^{-1/2}\Big] \leq \frac{x^2}{|x|^3} + \frac{1}{|x|} \leq \frac{C}{|x|^N}.
$$
\end{example}

In this paper, we work with multiresolution analyses (MRAs). Each MRA is characterized by a scaling function $\phi$ and a wavelet $\psi$ (for more on MRAs, see, for instance, Mallat \cite{mallat:1998} or Daubechies \cite{daubechies:1992}). In particular, the family of associated rescaled or shifted functions $\{\phi(t-k), 2^{j/2}\psi(2^jt - k), j \geq 0, k \in \bbZ \}$ constitutes an orthonormal basis of $L^2(\bbR)$. The rescalings of the function $\phi$ are related through a filter $u \in l^2(\bbZ)$ by means of the Fourier domain relation
\begin{equation}\label{e:phi_across_scales}
\widehat{\phi}( x ) = \frac{1}{\sqrt{2}}\widehat{u}\Big(\frac{x}{2}\Big)\widehat{\phi}\Big(\frac{x}{2}\Big),
\end{equation}
where the discrete Fourier transform (DFT) $\widehat{u}(x) = \sum_{n \in \bbZ} u_{n} e^{- i n x}$ of $u$ satisfies
\begin{equation}\label{e:sq_u^(x)+sq_u^(x+2pi)=2}
|\widehat{u}(x)|^2+ |\widehat{u}(x+ \pi)|^2=2.
\end{equation}
In turn, the relation between the two functions $\phi$, $\psi$ is given by
\begin{eqnarray}
\widehat{\psi}(x) = \frac{1}{\sqrt{2}} \hspace{1mm}
\widehat{v}\left(\frac{x}{2}\right)
\widehat{\phi}\left(\frac{x}{2}\right), \quad \widehat{v}(x) := e^{-ix} \overline{\widehat{u}(x + \pi)}.
\end{eqnarray}
The filters $u$ and $v$ are called the conjugate mirror filters (CMF) of the MRA. In this paper, two reference MRAs will be the Meyer and the Daubechies: the former, for comparison with previous works; the latter, for being a touchstone example of an MRA with non-compact Fourier domain support, thus serving as the primary working example for the extension efforts contained in this paper.

All through the manuscript, we will assume that the scaling and wavelet functions satisfy the assumptions that we lay out next.

\medskip
{\sc Assumption $(W1)$:}
\begin{equation}\label{e:general_assumption_decay_MRA}
|\phi(t)| = O((1 + t^2)^{-N/2 -1}), \quad |\psi(t)| = O((1 + t^2)^{-N/2 -1}),
\end{equation}
where $N$ is the number of vanishing moments of the wavelet $\psi$, i.e.,
\begin{equation}\label{e:def_N_zero_moments}
\int_{\bbR} t^{\nu} \psi(t) dt = 0, \quad \nu = 0 , 1, \hdots, N-1
\end{equation}
(see also condition \eqref{e:assumption_local_scaling}). Suppose also that
\begin{equation}\label{e:assumption_number_vanishing_moments}
N \geq \max\Big\{2,- \frac{1}{2} + |d| \Big\},
\end{equation}
where $d$ is as in \eqref{e:bounds-h_2}.

\medskip
{\sc Assumption $(W2)$:} for $N$ as above,
\begin{equation}\label{e:psihat_uhat_in_CN(R)}
\widehat{\phi} , \widehat{\psi} \in C^{N}(\bbR), \quad \widehat{u} \in C^{N}[-\pi,\pi).
\end{equation}
Assumptions $(W1)$ and $(W2)$ will allow us to make use of Theorem 7.4, p.\ 241 in Mallat \cite{mallat:1998}, which relates $N$ to the behavior of $\widehat{\psi}(x)$ and $\widehat{u}(x)$ around $x = 0$ and $\pi$, respectively. In fact, fix $A > 0$, and denote $\widehat{\psi}^{(\nu)}(x) =  \frac{d^{\nu}}{d x^{\nu}} (\Re \widehat{\psi}(x) + i \Im \widehat{\psi}(x))$, $\nu \geq 0$. By \eqref{e:general_assumption_decay_MRA}, \eqref{e:assumption_number_vanishing_moments}, \eqref{e:psihat_uhat_in_CN(R)} and the Taylor expansion with Lagrange residual for the real and imaginary parts of $\widehat{\psi}$, there exist functions $\lambda_1$, $\lambda_2$ on $[-A,A]$ such that
$$
\widehat{\psi}(x) = \Big( \frac{d^N}{dx^{N}} \Re \widehat{\psi}(x)\Big|_{\lambda_1(x)} + i \frac{d^N}{dx^{N}} \Im \widehat{\psi}(x)\Big|_{\lambda_2(x)} \Big)\frac{x^{N}}{N!}
$$
since $\widehat{\psi}^{(\nu)}(0) = 0$, $\nu = 0,\hdots,N-1$. Therefore, and by a similar reasoning applied to $\widehat{u}$,
\begin{equation}\label{e:psi^_around_zero_u^_around_pi}
|\widehat{\psi}(x)| = O(|x|^{N}), \quad |\widehat{u}(\pi+x)| = O(|x|^{N}), \quad x \rightarrow 0.
\end{equation}

The next assumption is a technical condition that is satisfied by most MRAs of interest. It is used in Lemma \ref{l:Phi0_gives_a_Riesz_basis_for_V0} below. Before we state it, we recall a definition from Daubechies \cite{daubechies:1992}, p.\ 182.

\begin{definition}\label{def:congruent_set}
A compact set $K$ is called congruent to $[-\pi,\pi]$ modulo $2\pi$ if $|K| = 2\pi$, and for $x \in [-\pi,\pi]$ there exists $l \in \bbZ$ such that $x + 2 l \pi \in K$.
\end{definition}

\medskip
{\sc Assumption $(W3)$:} For some compact set $K$ which is congruent to $[-\pi,\pi]$ modulo $2\pi$,
\begin{equation}\label{e:assumption_nontrivial_compact_set_in_support_phi}
\inf_{x \in K} |\widehat{\phi}(x)| > 0.
\end{equation}

\medskip

\begin{remark}
It can be seen that a Meyer scaling function $\widehat{\phi}$ can be made to satisfy \eqref{e:assumption_nontrivial_compact_set_in_support_phi}, since $\widehat{\phi}(x) = 1$ for $|x| \leq 2\pi/3$, $\widehat{\phi}(x) = 0$ for $|x| > 4\pi/3$ and $|\widehat{\phi}(x)|$ is decreasing on $[0,\infty)$. More generally, Theorem 6.3.1 in Daubechies \cite{daubechies:1992} shows that, under mild assumptions, a scaling function $\phi$ satisfies \eqref{e:assumption_nontrivial_compact_set_in_support_phi}  (see also p.\ 218, step 1 of the proof). In particular, this is true for a Daubechies scaling function.
\end{remark}

Furthermore, we will make use of the following relation, which connects the decay of $\widehat{\phi}$, $\widehat{\psi}$ to that of $h_2$ as expressed in \eqref{e:bounds-h_2}.

\medskip
{\sc Assumption $(W4)$:} there exist $\zeta, \eta > 0$ such that
\begin{equation}\label{e:assumption_decay_wavelet_Fourier}
|\widehat{\phi}(x)| \leq \frac{C}{(1 + |x|)^{|d|+ 1/2 + \zeta}}, \quad \max_{k=0,1,2}\Big|\frac{d^{k}}{dx^{k}}\widehat{\psi}(x)\Big| \leq \frac{C}{(1 + |x|)^{|d|+ (2 + \eta)}}, \quad x \in \bbR.
\end{equation}
\begin{remark}
Depending on the choice of the function $\widehat{u}$, a Meyer wavelet can be made arbitrarily smooth to satisfy \eqref{e:assumption_number_vanishing_moments}, \eqref{e:psihat_uhat_in_CN(R)} (see Mallat \cite{mallat:1998}, p.\ 247). Furthermore, it immediately satisfies \eqref{e:assumption_decay_wavelet_Fourier} due to its compact Fourier domain support.

Because of its compact time domain support, a Daubechies wavelet satisfies $\int_{\bbR} |\psi(t)| (1 + |t|^{\alpha})dt < \infty$ for any $\alpha \geq 1$. Consequently, $\widehat{\psi} \in C^{\infty}(\bbR)$ (see Daubechies \cite{daubechies:1992}, p.\ 216). Moreover, if the number of vanishing moments is large enough, then $\psi(t)$ is smooth and integration by parts can be used to establish \eqref{e:assumption_decay_wavelet_Fourier}. For instance, a Daubechies wavelet is in $C^2(\bbR)$ if its number of vanishing moments is at least 7 (see Daubechies \cite{daubechies:1992}, p.\ 232; a method that provides sharper estimates of the regularity as a function of the vanishing moments is provided in Daubechies and Lagarias \cite{daubechies:lagarias:1992:two}).
\end{remark}

The following is the main result of this work concerning the vaguelet property. See Section \ref{s:proof-vaguelet} for a proof.

\begin{theorem}\label{t:main-vaguelet}
Under assumption $(h2)$ and those for $\psi$ in $(W1)$, $(W2)$ and $(W4)$, both families of functions $\{\Psi_{j,k}, j \geq 0, k \in \bbZ \}$ and $\{\Psi^{j,k}, j \geq 0, k \in \bbZ \}$ are vaguelets.
\end{theorem}

As mentioned in the introduction, the theorem above can be complemented by showing that the families of functions (\ref{e:F_0}) and (\ref{e:F^0}) are also Riesz bases. For notational simplicity we will write
\begin{equation}\label{e:Phi_j_Phi^j_eta_j_eta^j}
\Phi_{j}= \frac{\Phi^{\#}_{j,0}}{\|\Phi^{\#}_{j,0} \|}, \quad \Phi^{j}= \frac{\Phi^{j,0}_{\#}}{\|\Phi^{j,0}_{\#} \|}, \quad \eta_j = \frac{\Psi^{\#}_{j,0}}{\| \Psi^{\#}_{j,0} \|}, \quad \eta^j = \frac{ \Psi^{j,0}_{\#} }{\|\Psi^{j,0}_{\#}\|}.
\end{equation}
The Riesz property of the collection \eqref{e:F_0} is equivalent to that of the collection
\begin{equation}\label{e:new-func-transform-normalized}
\{\Phi_{0}(t-k), \eta_{j}(t - 2^{-j}k),\ j \geq 0,k \in \bbZ \}.
\end{equation}
For future use, for $j \geq 0$, let
\begin{equation}\label{e:Vj_Wj}
V_j = \overline{\textnormal{span}} \{\Phi_j(t-2^{-j}k), k \in \bbZ\}, \quad W_j = \overline{\textnormal{span}} \{\eta_j(t-2^{-j}k), k \in \bbZ\}.
\end{equation}
Also, let
\begin{equation}\label{e:biorth-new-func-transform-normalized}
\{\Phi^{0}(t-k), \eta^{j}(t - 2^{-j}k),\ j \geq 0,k \in \bbZ \}
\end{equation}
be the collection of functions which are biorthogonal to
(\ref{e:new-func-transform-normalized}).

The following is the main result of this work concerning the Riesz property. See Section \ref{s:proof-riesz} for a proof.

\begin{theorem}\label{t:main}
Under assumptions $(h1)$--$(h3)$ and $(W1)$--$(W4)$, both families of functions \eqref{e:new-func-transform-normalized}, \eqref{e:biorth-new-func-transform-normalized} are Riesz bases of $L^2(\bbR)$.
\end{theorem}



\section{The vaguelet property}
\label{s:vaguelet}

This section concerns the vaguelet property of the families of functions $\{\Psi_{j,k}, j \geq 0, k \in \bbZ \}$ and $\{\Psi^{j,k}, j \geq 0, k \in \bbZ \}$. The vaguelet property is proved in Section \ref{s:proof-vaguelet} under the assumptions of Theorem \ref{t:main-vaguelet}. The necessity of the quasi-homogeneity assumption (\ref{e:quasihomogeneity}) is discussed in Section \ref{s:necessity_frac_decay}.

\subsection{Proof of Theorem \ref{t:main-vaguelet}}
\label{s:proof-vaguelet}

We first show that $\widehat{\Psi}^{\#}_{j,k}, \widehat{\Psi}^{j,k}_{\#} \in L^1(\bbR)$ (and, thus, that $\Psi^{\#}_{j,k}$, $\Psi^{j,k}_{\#}$ are well-defined pointwise), and that $\widehat{\Psi}^{\#}_{j,k}, \widehat{\Psi}^{j,k}_{\#}, \widehat{\Phi}^{\#}_{j,k}, \widehat{\Phi}^{j,k}_{\#} \in L^2(\bbR)$. The argument is based on the decay of the functions $\widehat{\psi}$ and $\widehat{\phi}$, respectively. The results on $\Phi^{\#}_{j,k}$, $\Phi^{j,k}_{\#}$ will be used in Section \ref{s:riesz}.
\begin{lemma}\label{l:bases_are_welldefined}
Under assumptions $(h1)$, $(h2)$, $(h3)$, $(W1)$, $(W2)$ and $(W4)$,
$$
\widehat{\Phi}^{\#}_{j,k}, \widehat{\Phi}^{j,k}_{\#} \in L^2(\bbR), \quad \widehat{\Psi}^{\#}_{j,k}, \widehat{\Psi}^{j,k}_{\#} \in L^1(\bbR) \cap L^2(\bbR).
$$
\end{lemma}
\begin{proof}
We first show that $\widehat{\Psi}^{\#}_{j,k} \in L^1(\bbR)$. Let $\epsilon > 0$. On the one hand, by \eqref{e:bounds-h_2} and \eqref{e:assumption_decay_wavelet_Fourier},
$$
\int_{|x| > \epsilon} |h_2(x)| \frac{1}{2^{j/2}}\Big| \widehat{\psi}\Big( \frac{x}{2^j}\Big)\Big| dx \leq C
\int_{|x| > \epsilon} \frac{1}{|x|^{d}} \frac{1}{2^{j/2}} \frac{1}{(1 + |2^{-j}x|)^{|d| + (2 + \eta)}}dx
$$
$$
\leq C' 2^{j(|d| + 2+ \eta -1/2)} \int_{|x| > \epsilon} \frac{1}{|x|^{d + |d| + (2 + \eta)}}dx < \infty.
$$

On the other hand, by \eqref{e:h2,h2inv_in_C2}, \eqref{e:assumption_number_vanishing_moments} and \eqref{e:psihat_uhat_in_CN(R)},
$$
\int_{|x| \leq \epsilon} |h_2(x)| \frac{1}{2^{j/2}}\Big| \widehat{\psi}\Big( \frac{x}{2^j}\Big)\Big| dx
\leq C \int_{|x| \leq \epsilon} \frac{1}{2^{j/2}}\Big| \widehat{\psi}\Big( \frac{x}{2^j}\Big)\Big| dx < \infty.
$$
Therefore, $\widehat{\Psi}^{\#}_{j,k} \in L^1(\bbR)$. By a similar argument, the same is true for $\widehat{\Psi}^{j,k}_{\#}$.

Also, a similar argument based on \eqref{e:h2,h2inv_in_C2}, \eqref{e:bounds-h_2}, \eqref{e:assumption_number_vanishing_moments}, \eqref{e:psihat_uhat_in_CN(R)} and \eqref{e:assumption_decay_wavelet_Fourier} establishes that $\widehat{\Psi}^{\#}_{j,k}, \widehat{\Psi}^{j,k}_{\#}  \in L^2(\bbR)$.

We now show that $\widehat{\Phi}^{\#}_{j,k} \in L^2(\bbR)$. For $A > 0$, by condition \eqref{e:assumption_h1_h1inv_bounded_compactinterv},
$$
\int_{|x| \leq A}|h_1(x)|^2 \frac{1}{2^j}\Big|\widehat{\phi}\Big(\frac{x}{2^j} \Big)\Big|^2 dx \leq C \int_{|x|\leq A}\frac{1}{2^j}\Big|\widehat{\phi}\Big(\frac{x}{2^j} \Big)\Big|^2 dx,
$$
which is finite by \eqref{e:assumption_number_vanishing_moments} and \eqref{e:psihat_uhat_in_CN(R)}. On the other hand, by \eqref{e:assumption_h1_h1inv_bounded_compactinterv}, \eqref{e:h2,h2inv_in_C2} and \eqref{e:assumption_h2/h1_is_periodic}, $h_1/h_2$ is periodic and bounded. Therefore, by \eqref{e:bounds-h_2} and \eqref{e:assumption_decay_wavelet_Fourier},
$$
\int_{|x|>A}|h_1(x)|^2 \frac{1}{2^j}\Big|\widehat{\phi}\Big(\frac{x}{2^j} \Big)\Big|^2  dx = \int_{|x|> A}\Big|\frac{h_1(x)}{h_2(x)}\Big|^2 |h_2(x)|^2 \frac{1}{2^j}\Big|\widehat{\phi}\Big(\frac{x}{2^j} \Big)\Big|^2  dx
$$
$$
\leq
C \int_{|x|>A} \frac{1}{|x|^{2d}} \frac{1}{2^j} \frac{1}{|2^{-j}x|^{2(|d|+1/2 + \zeta)}} \hspace{1mm}dx < \infty.
$$
%
%
An analogous reasoning shows that $\widehat{\Phi}^{j,k}_{\#} \in L^2(\bbR)$. $\Box$\\
\end{proof}

In the next proposition, we establish the asymptotic behavior of the norm of the functions
$\Psi^{\#}_{j,k}$ and $\Psi^{j,k}_{\#}$. The result will enter into the proofs of the vaguelet properties \eqref{e:Meyer_Coif_expression_5.1} and \eqref{e:Meyer_Coif_expression_5.3}.
\begin{proposition}\label{p:decay_norm}
Under assumptions $(h2)$, $(W1)$, $(W2)$, $(W4)$, there exist constants $0 < C_1 < C_2$, $0 < C^{'}_1 < C^{'}_2$ such that, for $j \geq 0$,
\begin{equation}\label{e:decay_Psi_sharp_jk_and_Psi_jk_sharp}
C_1 \leq \frac{\|\Psi^{\#}_{j,k}\|}{2^{-jd}} \leq C_2, \quad C^{'}_1 \leq \frac{\|\Psi^{j,k}_{\#}\|}{2^{jd}} \leq C^{'}_2.
\end{equation}
\end{proposition}
\begin{proof}
Consider $\Psi^{\#}_{j,k}$ first. Up to a constant, we can break up the square norm $\|\Psi^{\#}_{j,k}\|^2$ as
$$
\int_{\bbR} \Big| h_2(x) 2^{j/2} e^{-i 2^{-j}k x}\frac{1}{2^j}\widehat{\psi}\Big(\frac{x}{2^j}\Big)\Big|^2 dx = 2^{-j}\int_{\bbR} \Big| h_2(x)\widehat{\psi}\Big(\frac{x}{2^j}\Big)\Big|^2 dx
$$
\begin{equation}\label{e:square_norm_Psi_nonormalized}
= 2^{-j}\Big( \int_{|x|\leq \epsilon} +  \int_{|x| > \epsilon} \Big) \Big| h_2(x)\widehat{\psi}\Big(\frac{x}{2^j}\Big)\Big|^2 dx.
\end{equation}
Under \eqref{e:general_assumption_decay_MRA}, \eqref{e:assumption_number_vanishing_moments}, \eqref{e:psihat_uhat_in_CN(R)} (see also \eqref{e:psi^_around_zero_u^_around_pi}), the first term in the sum \eqref{e:square_norm_Psi_nonormalized} is
$$
2^{-j} \int_{|x|\leq \epsilon} | h_2(x)|^2 |O[(2^{-j}x)^N] |^2 dx
\leq C 2^{-j}\int_{|x|\leq \epsilon} |h_2(x)|^2 |(2^{-j}x)^N|^2 dx
$$
$$
\leq C 2^{-j(1+2N)}\int_{|x|\leq \epsilon} |h_2(x)|^2 x^{2N}dx \leq 2^{-j(1+2N)}  C(N),
$$
where the last inequality follows from the fact that $|h_2(x)|$ is bounded over $[-\epsilon,\epsilon]$ by \eqref{e:h2,h2inv_in_C2}. By \eqref{e:bounds-h_2}, the second term in \eqref{e:square_norm_Psi_nonormalized} can be bounded from above and below as
$$
C_1 2^{-j} \int_{|x| > \epsilon} |x|^{-2d} |\widehat{\psi}(2^{-j}x)|^2 dx \leq 2^{-j} \int_{|x| > \epsilon} |h_2(x)|^2 |\widehat{\psi}(2^{-j}x)|^2 dx \leq C_2 2^{-j} \int_{|x| > \epsilon} |x|^{-2d} |\widehat{\psi}(2^{-j}x)|^2 dx .
$$
In turn, a change of variables gives
$$
\int_{|x| > \epsilon} |2^{-j}x|^{-2d}2^{-2jd}|\widehat{\psi}(2^{-j}x)|^2 d(2^{-j}x)2^j
= 2^j 2^{-2jd} \int_{|y| > 2^{-j}\epsilon} |y|^{-2d}|\widehat{\psi}(y)|^2 dy
$$
$$
\leq 2^j 2^{-2jd}  \int_{\bbR} |y|^{-2d}|\widehat{\psi}(y)|^2 dy.
$$
The integral is finite (and does not depend on $j$) because of \eqref{e:general_assumption_decay_MRA}, \eqref{e:assumption_number_vanishing_moments}, \eqref{e:psihat_uhat_in_CN(R)} and \eqref{e:assumption_decay_wavelet_Fourier} (see also \eqref{e:psi^_around_zero_u^_around_pi}). Therefore,
$$
C_1 2^{-2jd} \leq 2^{-j} \int_{\bbR} \Big| h_2(x)\widehat{\psi}\Big(\frac{x}{2^j}\Big)\Big|^2 dx
\leq C_2 \Big( 2^{-j(1 + 2N)} + 2^{-2jd}\Big),
$$
which yields the first relation in \eqref{e:decay_Psi_sharp_jk_and_Psi_jk_sharp}. A similar development establishes the asymptotic behavior for $\|\Psi^{j,k}_{\#}\|$. $\Box$\\
\end{proof}


The next proposition establishes \eqref{e:Meyer_Coif_expression_5.3} for $\{ \Psi_{j,k} \}$ and $\{ \Psi^{j,k}\}$.
\begin{proposition}\label{p:Lip_vaguelet_general}
Under the assumptions $(h2)$, $(W1)$, $(W2)$, $(W4)$ the following Lipschitz property holds:
$$
\max\{ |\Psi_{j,k}(t) - \Psi_{j,k}(t')|,|\Psi^{j,k}(t) - \Psi^{j,k}(t')| \}\leq C 2^{j(1/2 + \alpha_2)}|t - t'|^{\alpha_2}, \quad 0 < \alpha_2 < 1.
$$
\end{proposition}
\begin{proof}
We adapt the argument in Mallat \cite{mallat:1998}, p.\ 165. We first consider $\Psi_{j,k}$. Without loss of generality, assume that $t \neq t'$. Then
$$
\frac{|\Psi^{\#}_{j,k}(t) - \Psi^{\#}_{j,k}(t')|}{|t-t'|^{\alpha_2}} = \frac{2^{j/2}}{2 \pi} \Big| \int_{\bbR} \frac{(e^{itx}- e^{it'x})}{|t -t'|^{\alpha_2}} e^{-i 2^{-j} k x }\frac{1}{2^j} \widehat{\psi}\Big(\frac{x}{2^j}\Big)h_2(x) dx \Big|
$$
\begin{equation}\label{e:Lipschitz_general_full_term}
\leq \frac{2^{j/2}}{2 \pi} \int_{\bbR} \frac{|e^{itx}- e^{it'x}|}{|t -t'|^{\alpha_2}}\frac{1}{2^j} \Big| \widehat{\psi}\Big(\frac{x}{2^j}\Big)\Big| |h_2(x)| dx ,
\end{equation}
where $\Psi^{\#}_{j,k}(t) - \Psi^{\#}_{j,k}(t')$ is well-defined by Lemma \ref{l:bases_are_welldefined}. We can break up the integral \eqref{e:Lipschitz_general_full_term} into two subdomains. On the one hand,
$$
\frac{2^{j/2}}{2 \pi} \int_{\{x:|t-t'|^{-1} \leq |x|\} } \frac{|e^{itx}- e^{it'x}|}{|t -t'|^{\alpha_2}}\frac{1}{2^j} \Big| \widehat{\psi}\Big(\frac{x}{2^j}\Big)\Big| |h_2(x)| dx
\leq \frac{2^{j/2}}{2 \pi} \int_{\{x:|t-t'|^{-1} \leq |x|\} } 2|x|^{\alpha_2}\frac{1}{2^j} \Big| \widehat{\psi}\Big(\frac{x}{2^j}\Big)\Big| |h_2(x)| dx
$$
$$
= \frac{2^{j(1/2 + \alpha_2)}}{2 \pi} 2\int_{\{ y:2^{-j}|t-t'|^{-1} \leq |y| \} } |y|^{\alpha_2} | \widehat{\psi}(y)| |h_2(2^{j}y)| dy.
$$
On the other hand,
$$
\frac{2^{j/2}}{2 \pi} \int_{\{x:|t-t'|^{-1} > |x|\} } \frac{|e^{itx}- e^{it'x}|}{|t -t'|^{\alpha_2}} \frac{1}{2^j} \Big| \widehat{\psi}\Big(\frac{x}{2^j}\Big)\Big| |h_2(x)| dx
\leq \frac{2^{j/2}}{2 \pi} \int_{\{x:|t-t'|^{-1} > |x|\} } |x|^{\alpha_2} \frac{1}{2^j} \Big| \widehat{\psi}\Big(\frac{x}{2^j}\Big)\Big| |h_2(x)| dx
$$
$$
= \frac{2^{j(1/2 + \alpha_2)}}{2 \pi} \int_{\{y:2^{-j}|t-t'|^{-1} > |y|\} } |y|^{\alpha_2} | \widehat{\psi}(y)| |h_2(2^{j}y)| dy.
$$
Therefore, \eqref{e:Lipschitz_general_full_term} is bounded by $2 \frac{2^{j(1/2 + \alpha_2)}}{2 \pi} \int_{\bbR} |y|^{\alpha_2} |h_2(2^{j}y)|  | \widehat{\psi}(y)| dy$. Consequently, by Proposition \ref{p:decay_norm},
$$
\frac{1}{\|\Psi^{\#}_{j,k}\|}  |\Psi^{\#}_{j,k}(t) - \Psi^{\#}_{j,k}(t')| \leq C \frac{1}{2^{-jd}} 2^{j(1/2 + \alpha_2)}|t-t'|^{\alpha_2}
\int_{\bbR}|\widehat{\psi}(y)| |h_2(2^j y)| |y|^{\alpha_2} dy.
$$

We now show that $\frac{1}{2^{-jd}} \int_{\bbR} |\widehat{\psi}(y)| |h_2(2^j y)| |y|^{\alpha_2}dy$ is bounded by a finite constant that does not depend on $j$.
By \eqref{e:assumption_local_scaling}, \eqref{e:general_assumption_decay_MRA}, \eqref{e:assumption_number_vanishing_moments} and \eqref{e:psihat_uhat_in_CN(R)},
$$
\frac{1}{2^{-jd}} \int_{|y| \leq \epsilon} |\widehat{\psi}(y)| |h_2(2^jy)| |y|^{\alpha_2} dy \leq C \int_{|y|\leq \epsilon} |\widehat{\psi}(y)| \frac{1}{|y|^{N-2}}|y|^{\alpha_2} dy
$$
$$
 = C \int_{|y|\leq \epsilon} | O(y^N) |\frac{1}{|y|^{N-2}}|y|^{\alpha_2} dy \leq C' < \infty.
$$
Moreover, by \eqref{e:bounds-h_2}, \eqref{e:assumption_decay_wavelet_Fourier},
\begin{equation}\label{e:Lip_property_integral_const_is_finite}
\frac{1}{2^{-jd}}\int_{|y| > \epsilon}|\widehat{\psi}(y)||h_2(2^jy)||y|^{\alpha_2} dy
\leq C \frac{1}{2^{-jd}}\int_{|y| > \epsilon} \frac{1}{(1 + |y|)^{|d|+2 + \eta}}\frac{1}{|2^j y|^d} |y|^{\alpha_2} dy.
\end{equation}
The right-hand side of \eqref{e:Lip_property_integral_const_is_finite} is independent of $j$. Furthermore, it is finite because $(|d| + d) + 2 + (\eta - \alpha_2) > 0 + 2 - 1 = 1$.

As for $\Psi^{j,k}$, the same type of argument leads to
$$
\frac{1}{\|\Psi^{j,k}_{\#}\|} |\Psi^{j,k}_{\#}(t) - \Psi^{j,k}_{\#}(t')| \leq \frac{C}{2^{jd}} 2^{j(1/2 + \alpha_2)}
|t-t'|^{\alpha_2} \int_{\bbR} |\widehat{\psi}(y)| |h^{-1}_2(2^j y)| |y|^{\alpha_2} dy.
$$
On the one hand, by \eqref{e:bounds-h_2} and \eqref{e:assumption_decay_wavelet_Fourier},
$$
\frac{1}{2^{jd}}\int_{|y|> \epsilon} |\widehat{\psi}(y)| |h^{-1}_2(2^j y)| |y|^{\alpha_2} dy \leq
\frac{C}{2^{jd}}\int_{|y|> \epsilon} |\widehat{\psi}(y)| |2^j y|^d |y|^{\alpha_2} dy
$$
$$
\leq \int_{|y|> \epsilon} \frac{1}{(1 + |y|)^{|d|+(2 + \eta)}} |y|^d |y|^{\alpha_2} dy \leq C < \infty,
$$
where finiteness holds by the same reasoning as that for \eqref{e:Lip_property_integral_const_is_finite}. On the other hand, by \eqref{e:assumption_local_scaling},
$$
\frac{1}{2^{jd}}\int_{|y| \leq \epsilon} |\widehat{\psi}(y)| |h^{-1}_2(2^j y)| |y|^{\alpha_2} dy \leq
C \int_{|y| \leq \epsilon} |\widehat{\psi}(y)| \frac{1}{|y|^{N-2}} |y|^{\alpha_2} dy \leq C' < \infty. \quad \Box
$$
\end{proof}

The next proposition establishes \eqref{e:Meyer_Coif_expression_5.1} for $\{ \Psi_{j,k} \}$ and $\{ \Psi^{j,k}\}$.
\begin{proposition}\label{p:decay_vaguelet_general}
If the assumptions $(h2)$, $(W1)$, $(W2)$, $(W4)$ hold, then
\begin{equation}\label{e:decay_vaguelet_general}
\max\{|\Psi_{j,k}(t)| ,|\Psi^{j,k}(t)|\} \leq \frac{C 2^{j/2}}{(1 + |2^j t - k|)^{1 + \alpha_1}}, \quad 0 < \alpha_1 < 1,
\end{equation}
for some $\alpha_1 > \alpha_2$, where $\alpha_2$ is as in Proposition \ref{p:Lip_vaguelet_general}.
\end{proposition}
\begin{proof}
By \eqref{e:h2,h2inv_in_C2}, \eqref{e:assumption_number_vanishing_moments} and \eqref{e:psihat_uhat_in_CN(R)}, we can apply a change of variables and integration by parts to obtain
$$
\Psi^{\#}_{j,k}(t) = \frac{2^{j/2}}{2\pi} \int_{\bbR} e^{itx} \frac{1}{2^j}e^{-ik2^{-j}x}h_2(x) \widehat{\psi}(2^{-j}x)dx
= \frac{2^{j/2}}{2\pi} \int_{\bbR} e^{i (2^j t - k) y} h_2(2^j y) \widehat{\psi}(y)dy
$$
$$
= \frac{2^{j/2}}{2\pi} \Big( \frac{e^{i(2^jt - k)y}}{i(2^j t - k)} [h_2(2^j y) \widehat{\psi}(y)]\Big|^{\infty}_{y = - \infty}
- \int_{\bbR} \frac{e^{i(2^jt - k)y}}{i(2^j t - k)} [2^j h^{'}_2(2^j y) \widehat{\psi}(y) + h_2(2^j y)\widehat{\psi}'(y)] dy \Big).
$$
The first term is zero due to \eqref{e:h2,h2inv_in_C2}, \eqref{e:bounds-h_2}, \eqref{e:assumption_decay_wavelet_Fourier}, and the second term is well-defined due to \eqref{e:h2,h2inv_in_C2}, \eqref{e:bounds-h_2}, \eqref{e:assumption_decay_h2_h2inv}, \eqref{e:assumption_decay_wavelet_Fourier}. Again by \eqref{e:h2,h2inv_in_C2}, \eqref{e:assumption_number_vanishing_moments}, \eqref{e:psihat_uhat_in_CN(R)} and integration by parts,
$$
\Psi^{\#}_{j,k}(t) = - \frac{2^{j/2}}{2\pi} \Big( \frac{e^{i(2^jt - k)y}}{[i(2^jt - k)y]^2} [2^j h^{'}_2(2^j y) \widehat{\psi}(y) + h_2(2^j y)\widehat{\psi}'(y)]\Big|^{\infty}_{y=-\infty}
 $$
\begin{equation}\label{e:integbyparts_secondorder_term}
 - \int_{\bbR}
\frac{e^{i(2^jt - k)y}}{[i(2^jt - k)y]^2} [h^{''}_2(2^jy)2^{2j} \widehat{\psi}(y) + 2 h^{'}_2(2^jy)2^j \widehat{\psi}'(y)
+ h_2(2^jy) \widehat{\psi}''(y)]dy \Big).
\end{equation}
The first term is zero due to \eqref{e:bounds-h_2}, \eqref{e:assumption_decay_h2_h2inv}, \eqref{e:assumption_decay_wavelet_Fourier}, and the second term is well-defined due to \eqref{e:h2,h2inv_in_C2}, \eqref{e:bounds-h_2}, \eqref{e:assumption_decay_h2_h2inv}, \eqref{e:assumption_number_vanishing_moments}, \eqref{e:psihat_uhat_in_CN(R)}, \eqref{e:assumption_decay_wavelet_Fourier}.

We look at each term of the sum \eqref{e:integbyparts_secondorder_term} separately, and break them up into the sum $\int_{|y| \leq \epsilon} + \int_{|y| > \epsilon}$. By \eqref{e:assumption_local_scaling}, \eqref{e:general_assumption_decay_MRA}, \eqref{e:assumption_number_vanishing_moments}, \eqref{e:psihat_uhat_in_CN(R)},
$$
\frac{1}{2^{-jd}} \int_{|y| \leq \epsilon}  |h_2(2^jy)| | \widehat{\psi}''(y)| dy
\leq C \int_{|y| \leq \epsilon}  \frac{1}{|y|^{N-2}} | O(y^{N-2})| dy \leq C' < \infty.
$$
On the other hand, by \eqref{e:bounds-h_2} and \eqref{e:assumption_decay_wavelet_Fourier},
$$
\frac{1}{2^{-jd}} \int_{|y| > \epsilon}  |h_2(2^jy)| | \widehat{\psi}''(y)| dy \leq C \int_{|y| > \epsilon}  |y|^{-d} \frac{1}{(1 + |y|)^{|d| + (2 + \eta)}} dy \leq C' < \infty.
$$
The remaining terms in \eqref{e:integbyparts_secondorder_term} can be treated similarly. This establishes the bound in \eqref{e:decay_vaguelet_general} for $\Psi_{j,k}(t)$ with $\alpha_1 = 1$, and hence also for $0 < \alpha_1 < 1$ with $\alpha_1$ greater than $\alpha_2 \in (0,1)$.

We now turn to $\Psi^{j,k}$. By a similar procedure, we obtain
$$
\Psi^{j,k}(t) = \frac{2^{j/2}}{2\pi}\frac{1}{2^{jd}} \int_{\bbR} \frac{e^{i(2^jt-k)x}}{[i(2^jt-k)y]^2}
[(h^{-1}_2)^{''}(2^jy)2^{2j} \widehat{\psi}(y) + 2 (h^{-1}_2)^{'}(2^jy)2^j \widehat{\psi}'(y)
+ h^{-1}_2(2^jy) \widehat{\psi}''(y)]dy .
$$
As done above, we look at each term of the sum separately, and break them up into the sum $\int_{|y| \leq \epsilon} + \int_{|y| > \epsilon}$. By \eqref{e:assumption_local_scaling}, \eqref{e:general_assumption_decay_MRA}, \eqref{e:assumption_number_vanishing_moments}, \eqref{e:psihat_uhat_in_CN(R)},
$$
\frac{1}{2^{jd}} \int_{|y| \leq \epsilon}  |h^{-1}_2(2^jy)| | \widehat{\psi}''(y)| dy
\leq C \int_{|y| \leq \epsilon}  \frac{1}{|y|^{N-2}} | O(y^{N-2})| dy \leq C' < \infty.
$$
On the other hand, by \eqref{e:bounds-h_2} and \eqref{e:assumption_decay_wavelet_Fourier},
$$
\frac{1}{2^{jd}} \int_{|y| > \epsilon}  |h_2(2^jy)| | \widehat{\psi}''(y)| dy \leq C \int_{|y| > \epsilon}  |y|^{d} \frac{1}{(1 + |y|)^{|d| + (2 + \eta)}} dy \leq C' < \infty.
$$
The remaining terms in the sum can be treated similarly. $\Box$\\
\end{proof}

We are now ready to conclude the proof of Theorem \ref{t:main-vaguelet}.\\

\noindent {\sc Proof of Theorem \ref{t:main-vaguelet}}: In view of Propositions \ref{p:Lip_vaguelet_general} and \ref{p:decay_vaguelet_general}, it only remains to show \eqref{e:Meyer_Coif_expression_5.2}. By \eqref{e:h2,h2inv_in_C2}, $h_2$ is bounded over compact intervals. Thus, for $N \geq 1$, $\widehat{\Psi}^{\#}_{j,k}(0) = 0$. Moreover, by Proposition \ref{p:decay_vaguelet_general}, $\Psi^{\#}_{j,k} \in L^1(\bbR)$. Thus, $\int_{\bbR}\Psi^{\#}_{j,k}(t) dt = 0$ (analogously for $\Psi^{j,k}_{\#}(t)$). $\Box$\\


\subsection{On quasi-homogeneity}\label{s:necessity_frac_decay}

In this section, we take up the issue raised in Remark \ref{r:exp(-x^2)}. In view of the framework laid out in Section \ref{s:assumptions-result}, one natural question is whether we can dispense with the quasi-homogeneity property \eqref{e:bounds-h_2}. Homogeneity assumptions or properties are commonly present in the wavelet analysis of stochastic processes (Meyer et al.\ \cite{MST:1999}, Benassi and Jaffard \cite{benassi:jaffard:1994}). We answer this question in the negative in the sense that, when \eqref{e:bounds-h_2} is violated, the induced family \eqref{e:new-func-transform-normalized} does not form vaguelets, even under Fourier domain compactness. We will provide a counterexample to the property \eqref{e:Meyer_Coif_expression_5.1} when the function $h_2$ in \eqref{e:psihat_uppersharp} takes the apparently simple exponential-like form
\begin{equation}\label{e:counterexample_exp_kernel}
h_2(x) = e^{-|x|^\gamma}, \quad \gamma \in (0,\infty).
\end{equation}
Throughout this section, we consider a wavelet function such that
\begin{equation}\label{e:even_wave}
\widehat{\psi}(x) = \widehat{\psi}(-x), \quad \widehat{\psi} \in \bbR,
\end{equation}
\begin{equation}\label{e:supp_psi_counterexample}
\textnormal{supp} \hspace{1mm} \{\widehat{\psi}\} \subseteq [-a,a],
\end{equation}
\begin{equation}\label{e:psihat_at_boundary}
\widehat{\psi} \in C^{r-1}(\bbR), \quad \widehat{\psi}(a) = \widehat{\psi}'(a) = \hdots = \widehat{\psi}^{(r-1)}(a) = 0, \quad \widehat{\psi}^{(r)}(a-) \neq 0, \quad r \in \bbN,
\end{equation}
for some $a > 0$. Condition \eqref{e:even_wave} is assumed for mathematical convenience. For instance, the Meyer wavelet can be written as $e^{ix} \widehat{\psi}(x)$, where $\widehat{\psi}$ satisfies \eqref{e:even_wave}. The term $e^{ix}$ can be incorporated as a shift in the developments below. Condition \eqref{e:psihat_at_boundary} is typically satisfied by a Meyer wavelet when $\widehat{u} \in C^{r-1}(\bbR)$ (see Mallat \cite{mallat:1998}, p.\ 247). In this case, $\widehat{\psi} \in C^{r-1}(\bbR)$, and the discontinuities of the $r$-th derivative of $\widehat{u}$ are generally at the boundary values $|x| = \pi/3, 2 \pi/3$.



For notational convenience, let
\begin{equation}\label{e:counterexample_decay_def_g}
g(x) = \frac{\widehat{\psi}(x)}{(a-x)^{r}}, \quad x \geq 0, \quad g(-x) = g(x).
\end{equation}
Then $\textnormal{supp} \hspace{1mm}\{g \}\subseteq [-a,a]$ and
$$
g(a-) \neq 0.
$$
Define the function $f_j$ in the Fourier domain as
\begin{equation}\label{e:fj}
\widehat{f}_{j}(x) = e^{2^{j \gamma}|x|^\gamma} \widehat{\psi}(x).
\end{equation}
Then we can write
$$
\widehat{\Psi}^{j,k}_{\#}(x) = e^{2^{j \gamma}|2^{-j}x|^\gamma} 2^{-j/2} e^{-i 2^{-j}x k} \widehat{\psi}(2^{-j}x) = 2^{j/2} \widehat{f_{j}(2^jt - k)}(x).
$$
In terms of $f_j$, the property (\ref{e:Meyer_Coif_expression_5.1}) amounts to
\begin{equation}\label{e:bound_fj_to_be_disproved}
\frac{|f_j(u)|}{\norm{f_j}} \leq \frac{C}{(1 + |u|)^{1 + \alpha}}.
\end{equation}
Our goal is to obtain the asymptotic behavior of $\|f_j\|$, and compare it to that of $f_j(u)$ for a special sequence of values of $u$. This is done in the next two lemmas.

For later reference, we write $a_j \sim b_j, j \rightarrow \infty$, when $\lim_{j \rightarrow \infty} \frac{a_j}{b_j} =1$.
\begin{lemma}\label{l:counterexample_decay_norm}
Let $f_j$ be as in \eqref{e:fj}. Then, for some $C > 0$,
\begin{equation}\label{e:counterexample_decay_norm}
\norm{f_j} \sim C e^{a^{\gamma} 2^{j\gamma}} 2^{-j\gamma(1+2r)/2}, \quad j \rightarrow \infty.
\end{equation}

\end{lemma}
\begin{proof}
For some constant $C> 0$,
$$
\norm{f_j}^2 = C \int_\bbR |\widehat{f}_j(x)|^2 dx = 2 C \int^{a}_0 e^{ 2^{j\gamma}2x^\gamma} |\widehat{\psi}(x)|^2 dx
$$
\begin{equation}\label{e:normsq_fj}
= 2 C e^{a^{\gamma} 2^{j \gamma +1}} \int^{a}_0 e^{-2^{j\gamma+1}(a^\gamma-x^\gamma)}|a-x|^{2r}g^2(x) dx,
\end{equation}
where $g$ is as in \eqref{e:counterexample_decay_def_g}. After the change of variables $2^{j \gamma}(a^\gamma-x^\gamma)=y$, the integral in (\ref{e:normsq_fj}) becomes
$$
e^{a^\gamma 2^{j\gamma+1}} 2^{-j \gamma} \int^{a^\gamma 2^{j\gamma}}_0 e^{-2y} \Big|a -\Big(a^\gamma - \frac{y}{2^{j\gamma}} \Big)^{1/\gamma} \Big|^{2r} g^2 \Big(\Big(a^\gamma - \frac{y}{2^{j\gamma}} \Big)^{1/\gamma} \Big) \frac{1}{\gamma}\Big(a^\gamma - \frac{y}{2^{j\gamma}}\Big)^{1/\gamma - 1} dy.
$$
\begin{equation}\label{e:norm_f_after_changevar}
=: e^{a^\gamma 2^{j\gamma+1}} 2^{-j \gamma} \frac{1}{\gamma} \Big( \int^{a^\gamma 2^{j\gamma}/2}_0 + \int^{a^\gamma 2^{j\gamma} - \delta_0}_{a^\gamma 2^{j\gamma}/2} + \int^{a^\gamma 2^{j\gamma}}_{a^\gamma 2^{j\gamma} - \delta_0} \Big) \rho_j(y) dy
\end{equation}
for some $\delta_0 > 0$ and large enough $j$. We now look at the first integral term in (\ref{e:norm_f_after_changevar}). By dominated convergence, we obtain
$$
\int^{a^\gamma 2^{j\gamma}/2}_0 \rho_j(y) dy =  \int^{a^\gamma 2^{j\gamma}/2}_{0} e^{-2y} \Big| \frac{y}{2^{j\gamma}} \Big|^{2r} \Big| \frac{(a^\gamma - y/2^{j\gamma})^{1/\gamma} - a}{(-y/2^{j\gamma})}\Big|^{2r} g^2 \Big(\Big(a^\gamma - \frac{y}{2^{j\gamma}} \Big)^{1/\gamma} \Big) \Big( a^\gamma - \frac{y}{2^{j\gamma}}\Big)^{1/\gamma - 1} dy
$$
$$
\sim \frac{1}{(2^{j\gamma})^{2r}} g^2(a-) \Big(\frac{1}{\gamma} \Big)^{2r} (a^{\gamma})^{1/\gamma - 1}\int^{\infty}_0 e^{-2y}y^{2r} dy,
$$
where the constants $g^2(a-)$, $(\frac{1}{\gamma})^{2r}$, $\int^{\infty}_0 e^{-2y}y^{2r} dy$ are strictly greater than zero. Furthermore, by making the change of variables $y = a^\gamma 2^{j\gamma}-w$,
$$
\int^{a^\gamma 2^{j\gamma}}_{a^\gamma 2^{j\gamma} - \delta_0} \rho_j(y)dy= \int^{a 2^{j\gamma}}_{a 2^{j\gamma}- \delta_0} e^{-2y} \Big| a - \Big( a^\gamma - \frac{y}{2^{j\gamma}}\Big)^{1/\gamma} \Big|^{2r} g^2\Big( \Big(a^\gamma - \frac{y}{2^{j\gamma}} \Big)^{1/\gamma} \Big) \Big( a^\gamma - \frac{y}{2^{j\gamma}}\Big)^{1/\gamma -1} dy
$$
\begin{equation}\label{e:counterexample_decay_norm_changevar}
= e^{-a^\gamma 2^{j\gamma+1}} \Big( \frac{1}{2^{j\gamma}} \Big)^{1/\gamma-1}\int^{\delta_0}_0 e^{2w} \Big| a - \Big( \frac{w}{2^{j\gamma}}\Big)^{1/\gamma} \Big|^{2r} g^2\Big( \Big( \frac{w}{2^{j\gamma}}\Big)^{1/\gamma}\Big) w^{1/\gamma -1}dw.
\end{equation}
By dominated convergence and \eqref{e:psihat_at_boundary}, the integral term in \eqref{e:counterexample_decay_norm_changevar} goes to a constant as $j \rightarrow \infty$. Consequently, the expression \eqref{e:counterexample_decay_norm_changevar} goes to zero faster than $1 / (2^{j\gamma})^{2r}$. Moreover, $\int^{a^\gamma 2^{j\gamma} - \delta_0}_{a^\gamma 2^{j\gamma}/2} \rho_j(y)dy$ can be written as
\begin{equation}\label{e:integral_between_2jr-delta0_and_2jr/2}
\frac{1}{(2^{j\gamma})^{2r}}
\int^{a^\gamma 2^{j\gamma} - \delta_0}_{a^\gamma 2^{j\gamma}/2} e^{-2y} |y|^{2r} \Big| \frac{(a^\gamma - y/2^{j\gamma})^{1/\gamma} - a}{(- y/2^{j\gamma})} \Big|^{2r}g^2\Big(\Big(
 a^\gamma - \frac{y}{2^{j\gamma}}\Big)^{1/\gamma} \Big) \Big(a^\gamma - \frac{y}{2^{j\gamma}}  \Big)^{1/\gamma - 1}dy
\end{equation}
Since the terms $| \frac{(a^\gamma - y/2^{j\gamma})^{1/\gamma} - a}{(- y/2^{j\gamma})} |^{2r}$ and $(a^\gamma- \frac{y}{2^{j\gamma}})^{1/\gamma - 1}$ are uniformly bounded in $j \in \bbN$ for $a^\gamma 2^{j\gamma}/2 \leq y \leq a^\gamma 2^{j\gamma} - \delta_0$, the integral in \eqref{e:integral_between_2jr-delta0_and_2jr/2} goes to zero as $j \rightarrow \infty$. Thus, (\ref{e:norm_f_after_changevar}) is asymptotically equivalent, up to a constant, to $e^{a^\gamma 2^{j\gamma+1}}2^{-j\gamma(1+2r)}$.
Consequently, \eqref{e:counterexample_decay_norm} holds. $\Box$
\end{proof}
\begin{lemma}\label{l:counterexample_fj_over_a_particular_seq}
Let $f_j$ be defined by \eqref{e:fj}. Then, for some $C > 0$,
\begin{equation}\label{e:counterexample_fj_over_a_particular_seq}
f_j\Big(\lfloor 2^{j \gamma}\rfloor \frac{2 \pi}{a} \Big) \sim C e^{a^\gamma 2^{j \gamma}} 2^{-j\gamma(1+r)}, \quad j \rightarrow \infty.
\end{equation}
\end{lemma}
\begin{proof}
First, we rewrite the expression for $f_j(u)$ in a more convenient way. For $u > 0$,
$$
f_j(u) = \frac{1}{2\pi} \int_{\bbR} e^{iux} e^{2^{j\gamma}|x|^\gamma} \widehat{\psi}(x) dx = \frac{1}{\pi} \Re \Big(\int^{a}_0 e^{iux} e^{ 2^{j\gamma}|x|^\gamma} \widehat{\psi}(x) dx \Big)
$$
$$
= \frac{1}{\pi}  e^{a^\gamma 2^{j\gamma}} \Re\Big( \int^{a}_0 e^{iux} e^{- 2^{j \gamma}(a^\gamma - x^\gamma)} \widehat{\psi}(x)dx \Big)
= \frac{1}{\pi}  e^{a^\gamma 2^{j\gamma}}\Re\Big(e^{iua} \int^{a}_0 e^{-i u (a-x)} e^{-2^{j\gamma} (a^\gamma - x^\gamma)}(a-x)^{r} g(x) dx \Big) .
$$
By the same change of variables as in \eqref{e:norm_f_after_changevar}, $f_j(u)$ becomes
$$
\frac{1}{\pi}  e^{a^\gamma 2^{j\gamma}} \Re \Big(e^{iua} \int^{a^\gamma 2^{j \gamma}}_0 e^{-iu(a - ( a^\gamma - y/2^{j\gamma})^{1/\gamma})} e^{-y} \Big( a - \Big( a^\gamma - \frac{y}{2^{j\gamma}}\Big)^{1/\gamma}\Big)^{r}
$$
\begin{equation}\label{e:asymptotics_fj(u)}
g \Big( \Big( a^\gamma - \frac{y}{2^{j\gamma}}\Big)^{1/\gamma}\Big)
\Big( a^\gamma - \frac{y}{2^{j\gamma}}\Big)^{1/\gamma -1 } \Big( -\frac{1}{\gamma}\Big) 2^{-j \gamma} dy\Big).
\end{equation}
Note that, over the sequence
\begin{equation}\label{e:counterexample_special_sequence}
u = \Big\{ \lfloor 2^{j\gamma} \rfloor \frac{2\pi}{a} \Big\}_{j \in \bbN},
\end{equation}
the exponential term $e^{iua}$ in \eqref{e:asymptotics_fj(u)} is identically 1. Thus, for $u$ as in \eqref{e:counterexample_special_sequence} and by applying a similar procedure to the one that leads to (\ref{e:counterexample_decay_norm}), we obtain \eqref{e:counterexample_fj_over_a_particular_seq}. $\Box$\\
\end{proof}

Lemmas \ref{l:counterexample_decay_norm} and \ref{l:counterexample_fj_over_a_particular_seq} yield the following result.
\begin{corollary} Let $f_j$ be defined by \eqref{e:fj}. Then, for some $C > 0$,
\begin{equation}\label{e:decay_fj(2jr2pi)/norm_fj}
\frac{f_{j}(\lfloor 2^{j\gamma} \rfloor 2\pi/a)}{\norm{f_j}} \sim C 2^{-j \gamma/2}.
\end{equation}
\end{corollary}
For $u$ as in \eqref{e:counterexample_special_sequence}, the right-hand side of \eqref{e:bound_fj_to_be_disproved} is asymptotically equivalent to $2^{-j\gamma(1 + \alpha)}$. This contradicts \eqref{e:decay_fj(2jr2pi)/norm_fj}, thus showing that \eqref{e:bound_fj_to_be_disproved}  does not hold. Thus, the functions in the family $\Psi^{j,k}$ do not satisfy \eqref{e:Meyer_Coif_expression_5.1} and are not vaguelets.


%
%

\section{The Riesz property}
\label{s:riesz}

This section concerns the Riesz property of the families of functions \eqref{e:new-func-transform-normalized} and \eqref{e:biorth-new-func-transform-normalized}. The Riesz property is proved in Section \ref{s:proof-riesz} under the assumptions of Theorem \ref{t:main}. The connection to the available results on the Riesz property of orthogonal wavelet bases in Sobolev spaces is discussed in Section \ref{s:Sobolev}.

\subsection{Proof of Theorem \ref{t:main}}
\label{s:proof-riesz}

By Theorem 2 on p.\ 56 of Meyer and Coifman \cite{coifman:meyer:1997}, an immediate consequence of Theorem \ref{t:main-vaguelet} is the following result which we state for later reference, and which is related to property $(i)$ of Riesz bases (see (\ref{e:Riesz_basis_equation_on_bounds})).

\begin{proposition}\label{p:ineq_vaguelet}
Under Assumptions $(h2)$, $(W1)$, $(W2)$ and $(W4)$, there exists a constant $C > 0$ such that for all sequence $\{d_{j,k}\}_{j \geq 0, k \in \bbZ} \in l^2$,
\begin{equation}\label{e:ineq_vaguelet}
\max \Big\{\Big\|\sum^{\infty}_{j=0} \sum_{k \in \bbZ} d_{j,k} \Psi_{j,k}\Big\|,\Big\|\sum^{\infty}_{j=0} \sum_{k \in \bbZ} d_{j,k} \Psi^{j,k}\Big\| \Big\}\leq C \Big(\sum^{\infty}_{j=0} \sum_{k \in \bbZ} |d_{j,k}|^2\Big)^{1/2}.
\end{equation}
\end{proposition}

We first establish the denseness property $(ii)$ of Riesz bases (see below (\ref{e:Riesz_basis_equation_on_bounds})). To prove that the spaces generated by \eqref{e:new-func-transform-normalized} and \eqref{e:biorth-new-func-transform-normalized} are dense in $L^2(\bbR)$, we will need the following lemma.

\begin{lemma} \label{l:Phij_and_etaj_to_Phij+1}

Suppose that the assumptions $(h1)$, $(h2)$, $(h3)$, $(W1)$, $(W2)$, $(W4)$ hold. For $j \in \bbZ$ and $\{a_k\}_{k \in \bbZ}$, $\{b_k\}_{k \in \bbZ}$ $\in l^2(\bbZ)$, there exist unique sequences $\{c_k\}_{k \in \bbZ}, \{\widetilde c_k\}_{k \in \bbZ} \in l^2(\bbZ)$ such that
\begin{equation}
\sum_k a_k \Phi_j(t - 2^{-j}k) + \sum_k b_k \eta_j(t - 2^{-j}k) =
\sum_k c_k \Phi_{j+1}(t - 2^{-(j+1)}k),
\label{e:Phij_and_etaj_to_Phij+1}
\end{equation}
\begin{equation}
\sum_k a_k \Phi^j(t - 2^{-j}k) + \sum_k b_k \eta^j(t - 2^{-j}k) =
\sum_k \widetilde c_k \Phi^{j+1}(t - 2^{-(j+1)}k).
\label{e:Phij_and_etaj_to_Phij+1_biorthog}
\end{equation}
Moreover, the respective induced maps from $l^2(\bbZ) \times l^2(\bbZ)$ to
$l^2(\bbZ)$ are isomorphisms.
\end{lemma}

\begin{proof} By Lemma \ref{l:bases_are_welldefined}, all the functions are well-defined in $L^2(\bbR)$.

We first show \eqref{e:Phij_and_etaj_to_Phij+1}. For this purpose, we adapt the proof of Lemma 5.3 in Meyer et al.\ \cite{MST:1999}. In terms of Fourier transforms,
(\ref{e:Phij_and_etaj_to_Phij+1}) may be expressed as
\begin{equation}
\widehat{a}\Big(\frac{x}{2^j}\Big) \widehat{\Phi}_j(x) +
\widehat{b}\Big(\frac{x}{2^j}\Big) \widehat{\eta}_j(x) =
\widehat{c}\Big(\frac{x}{2^{j+1}}\Big) \widehat{\Phi}_{j+1}(x),
\label{e:Phij_and_etaj_to_Phij+1_in_terms_of_FT}
\end{equation}
where $\widehat{a}$, $\widehat{b}$ and $\widehat{c}$ are,
respectively, the $2\pi$-periodic extensions of the discrete Fourier
transforms of $\{a_k\}$, $\{b_k\}$ and $\{c_k\}$. Define the functions
$$
U_j(x) = \frac{\|\Phi^{\#}_{j+1,0}\|}{\|\Phi^{\#}_{j,0}\|} \widehat{u}(2^{-(j+1)}x), \quad
V_j(x) =
 \frac{\|\Phi^{\#}_{j+1,0}\|}{\|\Psi^{\#}_{j,0}\|} \frac{h_2(x)}{h_1(x)} \widehat{v}(2^{-(j+1)}x),
$$
$x \in [-\pi,\pi)$, with the CMFs $u$ and $v$. Then
$$
\widehat{\Phi}_j(x) = U_j(x) \widehat{\Phi}_{j+1}(x), \quad
\widehat{\eta}_j(x) = V_j(x) \widehat{\Phi}_{j+1}(x).
$$
Therefore, the relation
(\ref{e:Phij_and_etaj_to_Phij+1_in_terms_of_FT}) may be rewritten as
$$
\widehat{a}(2x) U_j(2^{j+1}x) + \widehat{b}(2x) V_j(2^{j+1}x) =
\widehat{c}(x),
$$
which implies that $\widehat{c}$ can be obtained from $\widehat{a}$
and $\widehat{b}$. Moreover, by \eqref{e:assumption_h1_h1inv_bounded_compactinterv}, \eqref{e:h2,h2inv_in_C2}, \eqref{e:assumption_h2/h1_is_periodic}, $U_j(2^{j+1}x)$ and $V_j(2^{j+1}x)$ are $L^2[-\pi,\pi)$ functions, and
thus so is $\widehat{c}$.

Conversely, consider the family of matrices $\{M(x)\}_{x \in
[-\pi,\pi)}$, where
\[ M(x) = \begin{pmatrix}
   U_j(2^{j+1}x)     &     V_j(2^{j+1}x)        \\
                      &                          \\
 U_j(2^{j+1}(x+\pi))  &     V_j(2^{j+1}(x+\pi))
\end{pmatrix}. \]
These matrices are invertible, since
\begin{eqnarray}
\textnormal{det} [M(x)] & = & \frac{\|\Phi^{\#}_{j+1,0}\|}{\|\Phi^{\#}_{j,0}\|} \frac{{\|\Phi^{\#}_{j+1,0}\|} }{\|\Psi^{\#}_{j,0}\|} \hspace{0.1cm} \hspace{0.04cm}
\frac{h_2(2^{j+1}x)}{h_1(2^{j+1}x)} \hspace{0.04cm}
(\widehat{u}(x)\widehat{v}(x+\pi) - \widehat{u}(x+\pi)
\widehat{v}(x))\\ & = & \frac{\|\Phi^{\#}_{j+1,0}\|}{\|\Phi^{\#}_{j,0}\|} \frac{{\|\Phi^{\#}_{j+1,0}\|} }{\|\Psi^{\#}_{j,0}\|} \hspace{0.1cm} \frac{h_2(2^{j+1}x)}{h_1(2^{j+1}x)}
(-2e^{-ix}), \quad x\in [-\pi,\pi),
\end{eqnarray}
is bounded away from zero by \eqref{e:assumption_h1_h1inv_bounded_compactinterv} and \eqref{e:h2,h2inv_in_C2}. Thus,
$\widehat{a}$ and $\widehat{b}$ can also be recovered from
$\widehat{c}$ by using
$$
\begin{pmatrix}
   \widehat{a}(2x)      \\
  \widehat{b}(2x)
\end{pmatrix}
=
\begin{pmatrix}
   \widehat{a}(2(x+\pi))    \\
  \widehat{b}(2(x+\pi))
\end{pmatrix}
= M(x)^{-1}
\begin{pmatrix}
 \widehat{c}(x)     \\
\widehat{c}(x + \pi)
\end{pmatrix}.
$$
Moreover, since, by assumption, $\widehat{c} \in L^2[-\pi,\pi)$ and
$$
\widehat{a}(x) = \frac{\|\Phi^{\#}_{j+1,0}\|}{\| \Psi^{\#}_{j,0}\|} \frac{h_2}{h_1}(2^{j}x) \Big[ \widehat{v}\Big(\frac{x}{2}+\pi \Big) \widehat{c}\Big(\frac{x}{2}\Big) - \widehat{v}\Big(\frac{x}{2}\Big)
\widehat{c}\Big( \frac{x}{2}+ \pi \Big)\Big]\det\Big[M\Big( \frac{x}{2}\Big) \Big]^{-1},
$$
$$
\widehat{b}(x) = \frac{\|\Phi^{\#}_{j+1,0}\|}{\| \Phi^{\#}_{j,0}\|}  \Big[\widehat{u}\Big(\frac{x}{2}\Big)\widehat{c}\Big( \frac{x}{2}+ \pi\Big) - \widehat{u}\Big(\frac{x}{2} + \pi \Big)\widehat{c}\Big( \frac{x}{2} \Big)\Big]\det\Big[M\Big( \frac{x}{2}\Big) \Big]^{-1},
$$
%
%
then $\widehat{a}, \widehat{b}\in L^2[-\pi,\pi)$. To show \eqref{e:Phij_and_etaj_to_Phij+1_biorthog}, a similar argument can be developed with $\Phi^{j,0}_{\#}$, $\Psi^{j,0}_{\#}$, $U^j(x)$, $V^j(x)$, $\overline{h_2(x)^{-1}}$ and $\overline{h_1(x)^{-1}}$, instead. $\Box$\\
\end{proof}

The next proposition shows that the families \eqref{e:new-func-transform-normalized} and \eqref{e:biorth-new-func-transform-normalized} are, each one, dense in $L^2(\bbR)$. Its proof is different from that in Meyer et al.\ \cite{MST:1999} (see Theorem 1, p.\ 481) because we do not assume that the Fourier domain support of $\widehat{\psi}$ is compact. 

\begin{proposition}\label{p:denseness}
Under the assumptions $(h1)$, $(h2)$, $(h3)$, $(W1)$, $(W2)$, $(W4)$, the spans of the families \eqref{e:new-func-transform-normalized}, \eqref{e:biorth-new-func-transform-normalized} are dense in $L^2(\bbR)$.
\end{proposition}
\begin{proof}
Let $g \in L^2(\bbR)$. Since $\phi$, $\psi$ generate an orthogonal basis of $L^2(\bbR)$, we can take the sequences $\{a_{0,k}\}_{k \in \bbZ},
\{d_{j,k}\}_{j \geq 0, k \in \bbZ} \in l^2(\bbZ)$ as the coefficients of $g$ in that basis. In the Fourier domain, we can express
\begin{equation} \label{e:expansion_g_in_L2_in_wave}
\widehat{g}(x) = \widehat{a}_0(x) \widehat{\phi}(x) + \sum^{\infty}_{j=0} \widehat{d}_j (2^{-j}x)2^{-j/2}\widehat{\psi}(2^{-j}x),
\end{equation}
where $\widehat{a}_0(x)$, $\widehat{d}_j (x) $ denote, respectively, the discrete Fourier transforms of $\{a_{0,k}\}_{k \in \bbZ}$,
$\{d_{j,k}\}_{k \in \bbZ}$.

We first consider the family \eqref{e:new-func-transform-normalized}. We now break up the proof into two cases.

First, assume that $d \geq 0$. By \eqref{e:h2,h2inv_in_C2} and \eqref{e:bounds-h_2}, $h_2$ is bounded. Let $\widehat{d}^{(d)}_j (2^{-j}x):= \widehat{d}_j (2^{-j}x) \|\Psi^{\#}_{j,0}\|$ and $d^{(d)}_{j,k} = d_{j,k} \| \Psi^{\#}_{j,0}\|$. Since $d \geq 0$, then by Proposition \ref{p:decay_norm},
\begin{equation}\label{e:djk_2^(-jd)_in_l2}
\sum^{\infty}_{j=0}\sum_{k \in \bbZ} |d^{(d)}_{j,k}|^2 < \infty.
\end{equation}
Therefore, \eqref{e:expansion_g_in_L2_in_wave} can be written as
\begin{equation}\label{e:g^_expansion_introducing_h1}
\widehat{g}(x) = \widehat{a}_0(x) \frac{\widehat{\Phi}_0(x)}{h_1(x)} + \sum^{\infty}_{j=0} \widehat{d}^{(d)}_j (2^{-j}x)\frac{1}{\|\Psi^{\#}_{j,0}\|}2^{-j/2}\widehat{\psi}(2^{-j}x),
\end{equation}
with convergence in $L^2(\bbR)$. We claim that, as a consequence of \eqref{e:g^_expansion_introducing_h1}, we can write
\begin{equation}\label{e:g^_expansion_introducing_h2}
\widehat{g}(x)h_2(x) = \widehat{a}_0(x) \frac{h_2(x)}{h_1(x)}\widehat{\Phi}_0(x) + \sum^{\infty}_{j=0} \widehat{d}^{(d)}_j (2^{-j}x)\frac{1}{\|\Psi^{\#}_{j,0}\|}h_2(x) 2^{-j/2}\widehat{\psi}(2^{-j}x),
\end{equation}
where the equality holds in the $L^2(\bbR)$ sense. In fact, by the boundedness of $h_2(x)$, $\widehat{g}(x)h_2(x) \in L^2(\bbR)$. Also, by \eqref{e:assumption_h1_h1inv_bounded_compactinterv}, \eqref{e:h2,h2inv_in_C2} and \eqref{e:assumption_h2/h1_is_periodic}, the function $h_2/h_1$ is $2\pi$-periodic and bounded. Since $\Phi_0 \in L^2(\bbR)$, then the first term on the right-hand side of \eqref{e:g^_expansion_introducing_h2} is also in $L^2(\bbR)$. Moreover, by Proposition \ref{p:ineq_vaguelet} and \eqref{e:djk_2^(-jd)_in_l2}, the second term on the right-hand side of \eqref{e:g^_expansion_introducing_h2} converges in $L^2(\bbR)$. This establishes \eqref{e:g^_expansion_introducing_h2}.


In particular, let $f \in L^2_b(\bbR)$, where
\begin{equation}\label{e:L2b}
L^2_b(\bbR) = \{ f \in L^2(\bbR): \textnormal{\textnormal{supp}\{$\widehat{f}$\} is bounded away from 0 and $\pm \infty$}\}.
\end{equation}
By \eqref{e:h2,h2inv_in_C2} and \eqref{e:bounds-h_2}, $\widehat{g}(x) := \widehat{f}(x)/h_2(x) \in L^2(\bbR)$. Therefore, the relation \eqref{e:g^_expansion_introducing_h2} holds with $\widehat{g}(x) h_2(x) = (\widehat{f}(x)/h_2(x))h_2(x) = \widehat{f}(x)$ on the left-hand side. The claim now follows from the fact that $\overline{L^2_b(\bbR)} = L^2(\bbR)$.

Alternatively, assume that $d < 0$. Once again, our goal is to show that \eqref{e:g^_expansion_introducing_h2} holds, where equality is in $L^2(\bbR)$. However, this time we impose the stronger assumption that $\widehat{g}$ lies in the Sobolev space $W^d(\bbR)$ (see \eqref{e:def_H^s} below). Then, by Proposition \ref{p:decay_norm} and Daubechies \cite{daubechies:1992}, pp.\ 298-299,
\begin{equation}\label{e:summability_djk_2^-jd_d<0}
\sum^{\infty}_{j=0}\sum_{k \in \bbZ}|d_{j,k}|^2 \|\Psi^{\#}_{j,k}\|^2  < \infty.
\end{equation}
In particular, for $d^{(d)}_{j,k} := d_{j,k}\|\Psi^{\#}_{j,k}\|$,
\eqref{e:summability_djk_2^-jd_d<0} amounts to the summability condition \eqref{e:djk_2^(-jd)_in_l2}. Now proceed as in the case $d > 0$ to establish \eqref{e:g^_expansion_introducing_h2} in the $L^2(\bbR)$ sense.

For $f \in L^{2}_{b}(\bbR)$, we claim that $\widehat{g}(x)(1 + |x|^2)^{-d/2} := (\widehat{f}(x)/h_2(x))(1 + |x|^2)^{-d/2} \in L^2(\bbR)$. In fact, by \eqref{e:bounds-h_2},
$$
\int_{|x| > \epsilon} \Big| \frac{\widehat{f}(x)}{h_2(x)} \Big|^2 (1 + |x|^2)^{-d} dx \leq C \int_{|x| > \epsilon} \Big| \frac{\widehat{f}(x)}{|x|^{-d}} \Big|^2 (1 + |x|^2)^{-d} dx < \infty.
$$
Moreover, again by \eqref{e:h2,h2inv_in_C2}, $\int_{|x| \leq \epsilon} | \frac{\widehat{f}(x)}{h_2(x)} |^2 (1 + |x|^2)^{-d} dx < \infty$.

For \eqref{e:biorth-new-func-transform-normalized}, one can use a similar argument with $\Psi^{j,k}_{\#}$, $\widehat{\Phi}^0$, $\overline{h_{1}(x)^{-1}}$, $\overline{h_2(x)^{-1}}$. $\Box$\\
\end{proof}

The next lemma will be used in the proof of Theorem \ref{t:main}. It establishes that the shifts of $\Phi_0$, $\Phi^{0}$ generate a Riesz basis of its own span.

\begin{lemma} \label{l:Phi0_gives_a_Riesz_basis_for_V0}
Suppose that the assumptions $(h1)$, $(h2)$, $(h3)$, $(W3)$, $(W4)$ hold. Then the families $\{\Phi_0(t - k)\}_{k \in \bbZ}$, $\{\Phi^0(t - k)\}_{k \in \bbZ}$ are Riesz bases of their closed linear spans $V_0$, $V^0$, respectively, in $L^2(\bbR)$.
\end{lemma}
\begin{proof}
We first show that $\{\Phi_0(t - k)\}_{k \in \bbZ}$ is a Riesz basis for $V_0$.
It suffices to show that there exist $C_1, C_2 > 0$ such that
\begin{equation}\label{e:Daubechies_condition_for_Riesz_basis_of_own_span}
0 < C_1 \leq \sum_{k \in \bbZ} | \widehat{\Phi}_0(x + 2 \pi k)|^2 \leq C_2 < \infty
\end{equation}
(see Daubechies \cite{daubechies:1992}, p.\ 140). Since the function $\sum_{k \in \bbZ} | \widehat{\Phi}_0(x + 2 \pi k)|^2$ has period $2\pi$, it suffices to consider $x$ in the compact set $K$ as in Definition \ref{def:congruent_set}. By \eqref{e:assumption_h1_h1inv_bounded_compactinterv} and \eqref{e:assumption_nontrivial_compact_set_in_support_phi},
$$
\sum_{k \in \bbZ}|\widehat{\Phi}_0 (x + 2 \pi k)|^2 \geq \sum_{k \in \bbZ}|h_1 (x + 2 \pi k) \widehat{\phi} (x + 2 \pi k)|^2
\geq \inf_{x \in K} |h_1 (x)|^2 \inf_{x \in K}|\widehat{\phi} (x)|^2 >0.
$$
This establishes the lower bound in \eqref{e:Daubechies_condition_for_Riesz_basis_of_own_span}.

Note that, since the shifts of $\phi$ are a Riesz basis of the closure of their own span,
\begin{equation}\label{e:phi_is_Riesz}
0 < C \leq \sum_{k \in \bbZ} | \widehat{\phi}(x + 2 \pi k)|^2 \leq C' < \infty
\end{equation}
for some $C, C'> 0$. So, write
$$
\widehat{\Phi}_0 (x + 2 \pi k) = \frac{h_1(x + 2 \pi k )}{h_2(x + 2\pi k)} h_2(x + 2 \pi k) \widehat{\phi}(x + 2 \pi k).
$$
By \eqref{e:assumption_h1_h1inv_bounded_compactinterv}, \eqref{e:h2,h2inv_in_C2} and \eqref{e:assumption_h2/h1_is_periodic}, the ratio $|\frac{h_1(x )}{h_2(x)}|$ can be bounded from above and below uniformly in $x \in \bbR$. Therefore, for some $C > 0$,
\begin{equation}\label{e:Phi_0_first_bound}
|\widehat{\Phi}_0 (x + 2 \pi k)|^2 \leq  C' |h_2(x + 2 \pi k) \widehat{\phi}(x + 2 \pi k)|^2.
\end{equation}
Assume that $d < 0$. Therefore, by \eqref{e:assumption_h1_h1inv_bounded_compactinterv}, \eqref{e:h2,h2inv_in_C2}, \eqref{e:assumption_decay_wavelet_Fourier} and \eqref{e:Phi_0_first_bound}, for some $C', C'' > 0$ we have
$$
\sum_{k \in \bbZ}|\widehat{\Phi}_0 (x + 2 \pi k)|^2 \leq C' \sum_{k \in \bbZ}\Big|(1 + |x+ 2\pi k|^{-d}) \Big(\frac{1}{1 + |x + 2 \pi k|^{-d + \frac{1}{2} + \zeta}}\Big)\Big|^2 \leq C''.
$$
Alternatively, if $d \geq 0$, by \eqref{e:assumption_h1_h1inv_bounded_compactinterv}, \eqref{e:h2,h2inv_in_C2}, the upper bound in \eqref{e:Daubechies_condition_for_Riesz_basis_of_own_span} results from \eqref{e:phi_is_Riesz}.

To show that $\{\Phi^0(t - k)\}_{k \in \bbZ}$ is a Riesz basis for $V^0$, a similar argument can be developed with $\widehat{\Phi}_0$, $\overline{h_1(x)^{-1}}$, $\overline{h_2(x)^{-1}}$, instead. $\Box$\\
\end{proof}
%
%

We are now in position to prove the main result of the paper. The argument draws upon that in Meyer et al.\ \cite{MST:1999}, pp.\ 480-481. \\

{\sc Proof of Theorem \ref{t:main}:}
We first show it for \eqref{e:new-func-transform-normalized}. Let $V_j$ and $W_j$ be as in \eqref{e:Vj_Wj}. By Lemma
\ref{l:Phij_and_etaj_to_Phij+1},
\begin{equation}
V_j \oplus W_j = V_{j+1}, \label{e:Vj_+Wj}
\end{equation}
which is a direct but not orthogonal sum. Therefore, by Proposition \ref{p:denseness},
\begin{equation}
V_j \subseteq V_{j+1}, \quad \bigcap^{\infty}_{j=0} V_j = V_0, \quad
\textnormal{and} \quad \bigcup^{\infty}_{j=0} V_j \textnormal{ is
dense in } L^2(\bbR) \label{e:union_Vj_is_dense_in_L^2}.
\end{equation}
It follows from (\ref{e:Vj_+Wj}) and
(\ref{e:union_Vj_is_dense_in_L^2}) that the space $V_0 \bigoplus_{j \geq
0} W_j$ is dense in $L^2(\bbR)$, which gives part $(ii)$ of the
definition of a Riesz basis. We now turn to showing part $(i)$ of the definition.

Suppose for the moment that there exist constants $C_2 \geq C_1 > 0$
such that
\begin{equation}
C_1 \Big(\sum_{j \geq 0}\sum_{k} b_{j,k}^2 \Big)^{1/2} \leq
\Big\|\sum_{j \geq 0} \sum_{k} b_{j,k} \eta_{j}(t - 2^{-j}k)
\Big\| \leq C_2 \Big(\sum_{j \geq 0}\sum_{k} b_{j,k}^2
\Big)^{1/2} \label{e:bounds_above_below_for_eta}
\end{equation}
for any sequence $\{b_{j,k}\} \in l^2(\bbZ)$, and that the same relation
holds for $\eta^j$. Then, since the family
$\{\Phi_0 (t-k), k \in \bbZ\}$ is a Riesz basis of $V_0$ by Lemma
\ref{l:Phi0_gives_a_Riesz_basis_for_V0}, we have
$$
\Big\|\sum_{k}a_k \Phi_0 (t-k) + \sum_{j \geq 0}\sum_{k}
b_{j,k} \eta_j (t-2^{-j}k)\Big\|
$$
\begin{equation}\label{e:ineq_sum_approx_details}
\leq \sqrt{2} \Big( \Big\|\sum_{k}a_k \Phi_0 (t-k)\Big\|^2 +
\Big\| \sum_{j \geq 0}\sum_{k} b_{j,k} \eta_j (t-2^{-j}k)
\Big\|^2 \Big)^{1/2} \leq C \hspace{1mm} \Big( \sum_{k}a^{2}_k +
\sum_{j \geq 0} \sum_{k} b^{2}_{j,k} \Big)^{1/2},
\end{equation}
for some constant $C$, which establishes the right-hand side inequality of
(\ref{e:Riesz_basis_equation_on_bounds}). The left-hand side inequality of
(\ref{e:Riesz_basis_equation_on_bounds}) may be shown in the following
way. By \eqref{e:biortho_relation} and \eqref{e:orthog_of_Phi_and_Psi}, the functions $\Phi^0(t-k), \Phi_0(t-k),
\eta^j(t-2^{-j}k),\eta_j(t-2^{-j}k), k \in \bbZ, j \geq 0$, satisfy
the relations
\begin{equation}\label{e:biorthogonality_Phi^0_eta^j}
\int_{\bbR} \Phi^0(t-k)\Phi_0(t-k') dt = \delta_{\{k=k'\}}, \quad
\int_{\bbR} \eta^j(t-2^{-j}k) \eta_{j'}(t-2^{-j'}k') dt =
\delta_{\{j=j',k=k'\}},
\end{equation}
$$
\int_{\bbR} \Phi^0(t-k)\eta_{j}(t-2^{-j}k') dt = 0 \quad
\textnormal{and} \quad \int_{\bbR} \eta^{j}(t-2^{-j}k)\Phi_0(t-k')
dt = 0, \quad j\geq 0.
$$
Then, for any sequences $\{a_{k}\}, \{b_{j,k}\} \in l^{2}(\bbZ)$,
$$
\sum_{k} a^{2}_{k} + \sum_{j \geq 0}\sum_{k} b^{2}_{j,k} =
\int_{\bbR} \Bigg(\sum_{k} a_{k}\Phi^{0}(t - k) + \sum_{j
\geq 0}\sum_{k}b_{j,k} \eta^{j}(t- 2^{-j}k) \Bigg)\cdot
$$
$$
.\Bigg(\sum_{k} a_{k}\Phi_{0}(t - k) + \sum_{j \geq 0}\sum_{k}
b_{j,k} \eta_{j}(t- 2^{-j}k) \Bigg) dt
$$
$$
\leq \Big\| \sum_{k} a_{k}\Phi_{0}(t - k) + \sum_{j
\geq 0} \sum_{k} b_{j,k} \eta_{j}(t- 2^{-j}k) \Big\| \cdot \Big\|
\sum_{k} a_{k}\Phi^{0}(t - k) + \sum_{j \geq 0} \sum_{k} b_{j,k}
\eta^{j}(t- 2^{-j}k) \Big\|
$$
$$
\leq C \Big(\sum_{k } a^{2}_{k} + \sum_{j \geq 0}\sum_{k}
b^{2}_{j,k} \Big)^{1/2} \Big\| \sum_{k} a_{k}\Phi_{0}(t - k) +
\sum_{j \geq 0} \sum_{k} b_{j,k} \eta_{j}(t- 2^{-j}k) \Big\|
$$
for some constant $C > 0$, and hence
\begin{equation} \label{e:Riesz_proof_left_bound_Phi}
\frac{1}{C} \Big( \sum_{k}  a^{2}_{k} +
 \sum_{j \geq 0}\sum_{k} b^{2}_{j,k} \Big)^{1/2} \leq \Big\|
 \sum_{k}  a_{k}\Phi_{0}(t - k) +
\sum_{j \geq 0} \sum_{k}b_{j,k} \eta_{j}(t- 2^{-j}k) \Big\|.
\end{equation}

It remains to prove (\ref{e:bounds_above_below_for_eta}). The upper bound is established for both $\Psi_{j,k}$, $\Psi^{j,k}$ in Proposition \ref{p:ineq_vaguelet}. As for the lower bound, by \eqref{e:biorthogonality_Phi^0_eta^j} and Proposition \ref{p:ineq_vaguelet},
$$
\sum_{j \geq 0}\sum_{k \in \bbZ} b^{2}_{j,k} = \int_{\bbR} \Big( \sum_{j \geq 0}\sum_{k \in \bbZ} b_{j,k} \Psi_{j,k}(t) \Big)
\Big(\sum_{j' \geq 0}\sum_{k' \in \bbZ} b_{j',k'} \Psi^{j',k'}(t)\Big) dt
$$
$$
\leq \Big\| \sum_{j \geq 0}\sum_{k \in \bbZ} b_{j,k} \Psi_{j,k} \Big\| \Big\| \sum_{j' \geq 0} \sum_{k' \in \bbZ} b_{j',k'} \Psi^{j',k'} \Big\|
\leq  \Big\| \sum_{j \geq 0}\sum_{k \in \bbZ} b_{j,k} \Psi_{j,k} \Big\| C \Big( \sum_{j' \geq 0} \sum_{k' \in \bbZ} b^{2}_{j',k'}  \Big)^{1/2}.
$$

To show that \eqref{e:biorth-new-func-transform-normalized} are a Riesz basis of $L^2(\bbR)$, a similar argument can be developed with $V^j$ and $W^j$, instead. $\Box$\\


\subsection{Connection with the Riesz property in Sobolev spaces}
\label{s:Sobolev}


It is possible to establish the Riesz property for our bases in $L^2(\bbR)$ by taking a direct approach via Sobolev spaces, i.e., without proving that the bases form vaguelets. In this section, we outline how this can be done in light of recent results found in the literature (e.g., Han and Shen \cite{han:shen:2006,han:shen:2009} and references therein).

Let
\begin{equation}\label{e:def_H^s}
W^d(\bbR) = \Big\{f : \| f \|^2_{W^d(\bbR)}= \int_{\bbR} |\widehat{f}(x)|^2 (1 + |x|^2)^{d} dx < \infty \Big\}, \quad d \in \bbR,
\end{equation}
be the (Sobolev) space of functions.
The factor $(1 + |x|^2)^{d}$ above naturally connects with the bases investigated in this work. Indeed, under the assumptions \eqref{e:h2,h2inv_in_C2} and \eqref{e:bounds-h_2}, there are constants $c_2 \geq c_1>0$ such that
\begin{equation}\label{e:h2-sobolev}
    c_1 \leq \frac{h_2(x)}{(1+|x|^2)^{d/2}} \leq c_2,\quad \mbox{for all}\ x\in \bbR.
\end{equation}

Han and Shen \cite{han:shen:2009} introduced a novel approach for the study of the Riesz property that involves generating two families of wavelet-based functions $\{2^{-jd}\psi_{j,k}\}$ and $\{2^{jd}\widetilde{\psi}_{j,k}\}$. Instead of requiring both systems to be Riesz bases of $L^2(\bbR)$, and thus imposing the same assumptions on the properties of both, one of them can have the desired order of regularity while the other have the desired number of vanishing moments. Moreover, they are biorthogonal with respect to the inner product $\langle f,g\rangle := (2\pi)^{-1}\int_{\bbR}\widehat{f}(x)\overline{\widehat{g}(x)}dx$, $f \in W^d(\bbR)$, $g \in W^{-d}(\bbR)$, and, most importantly, are Riesz bases for the appropriate Sobolev spaces.

More precisely, under mild assumptions on the underlying so-called masks, the authors showed that $\{2^{-jd}\psi_{j,k}\}$ is a Riesz basis for $W^d(\bbR)$, i.e.,
\begin{itemize}
\item [\textit{(i')}] for $0 < C_1 \leq C_2$,
\begin{equation}\label{e:bounds_Riesz_Sobolev}
C_1 \sum^{\infty}_{j,k = -\infty} |d_{j,k}|^2 \leq \Big\| \sum^{\infty}_{j,k = -\infty} d_{j,k}2^{-jd}\psi_{j,k} \Big\|^2_{W^d(\bbR)} \leq C_2 \sum^{\infty}_{j,k = -\infty} |d_{j,k}|^2
\end{equation}
for compactly supported sequences  $\{d_{j,k}\}$;
\item [\textit{(ii')}] $\textnormal{span}\{\psi_{j,k}\}$ is dense in $W^d(\bbR)$,
\end{itemize}
and similarly that $\{2^{jd}\widetilde{\psi}_{j,k}\}$ is a Riesz basis for $W^{-d}(\bbR)$. Note that we do not include approximation terms in (\ref{e:bounds_Riesz_Sobolev}) and in the basis $\{2^{-jd}\psi_{j,k}\}$. These terms are included in Han and Shen \cite{han:shen:2009} with index $j$ ranging only from $0$ to $\infty$, but they cannot be related directly to our approximation functions $\Phi^\#_{j,k}$.

Taking $\{d_{j,k}\} \in l^2$ in (\ref{e:bounds_Riesz_Sobolev}) and using (\ref{e:h2-sobolev}) yield
\begin{equation}\label{e:bounds_Riesz_Sobolev_2}
C_1 \sum^{\infty}_{j,k = -\infty} |d_{j,k}|^2 \leq \Big\| \sum^{\infty}_{j,k = -\infty} d_{j,k}2^{-jd}\Psi^\#_{j,k} \Big\|^2_{L^2(\bbR)} \leq C_2 \sum^{\infty}_{j,k = -\infty} |d_{j,k}|^2,
\end{equation}
where $\Psi^\#_{j,k}$ are defined in (\ref{e:psihat_uppersharp}). In view of Proposition \ref{p:decay_norm} (and supposing that this proposition extends to $j<0$), (\ref{e:bounds_Riesz_Sobolev_2}) is equivalent to the Riesz property $(i)$ of the family of functions $\{\Psi_{j,k}\}$. The denseness condition $(ii)$ of the Riesz property can be deduced similarly.

\bibliography{riesz}

\small

\bigskip

\noindent \begin{tabular}{lll}
Gustavo Didier & St\'{e}phane Jaffard & Vladas Pipiras \\
Mathematics Department  & Universit\'{e} Paris-Est  & Dept.\ of Statistics and Operations Research \\
Tulane University & Cr\'{e}teil Val-de-Marne  & UNC at Chapel Hill \\
6823 St.\ Charles Avenue  & B\^{a}timent P2, Etage 3, bureau 351  & CB\#3260, Smith Bldg. \\
New Orleans, LA 70118, USA & 61 avenue du G\'{e}n\'{e}ral de Gaulle  & Chapel Hill, NC 27599, USA \\
& 94010, Cr\'{e}teil, France  & \\
{\it gdidier@tulane.edu}& {\it stephane.jaffard@u-pec.fr}  & {\it pipiras@email.unc.edu} \\
\end{tabular}\\

\smallskip

\end{document}